\documentclass[reqno]{amsart}

\usepackage{amsmath}
\usepackage{amsfonts}
\usepackage{amssymb,enumerate}
\usepackage{amsthm}
\usepackage[all]{xy}
\usepackage{rotating}
\usepackage{hyperref}
\usepackage{color}
\theoremstyle{plain}
\newtheorem{lem}{Lemma}[section]
\newtheorem{cor}[lem]{Corollary}
\newtheorem{prop}[lem]{Proposition}
\newtheorem{thm}[lem]{Theorem}

\theoremstyle{definition}
\newtheorem{defn}[lem]{Definition}
\newtheorem{ex}[lem]{Example}

\newtheorem{disc}[lem]{Remark}

\newtheorem{fact}[lem]{Fact}

\newcommand{\cat}[1]{\mathcal{#1}}

\newcommand{\catd}{\cat{D}}

\newcommand{\pd}{\operatorname{pd}}

\newcommand{\id}{\operatorname{id}}	
\newcommand{\fd}{\operatorname{fd}}

\newcommand{\depth}{\operatorname{depth}}	
	
\newcommand{\amp}{\operatorname{amp}}

\newcommand{\mspec}{\operatorname{m-Spec}}

\newcommand{\HH}{\operatorname{H}}
\newcommand{\Hom}{\operatorname{Hom}}	

\newcommand{\spec}{\operatorname{Spec}}

\newcommand{\shift}{\mathsf{\Sigma}}

\newcommand{\ideal}[1]{\mathfrak{#1}}
\newcommand{\m}{\ideal{m}}

\newcommand{\p}{\ideal{p}}
\newcommand{\q}{\ideal{q}}

\newcommand{\fa}{\ideal{a}}
\newcommand{\fb}{\ideal{b}}

\newcommand{\supp}{\operatorname{supp}}
\newcommand{\Supp}{\operatorname{Supp}}

\newcommand{\VE}{\operatorname{V}}
\newcommand{\cosupp}{\operatorname{co-supp}}

\newcommand{\bbz}{\mathbb{Z}}

\newcommand{\xra}{\xrightarrow}

\newcommand{\vf}{\varphi}

\newcommand{\y}{\mathbf{y}}

\newcommand{\x}{\underline{x}}

\renewcommand{\geq}{\geqslant}
\renewcommand{\leq}{\leqslant}

\newcommand{\Rhom}[3][R]{\mathbf{R}\!\operatorname{Hom}_{#1}(#2,#3)}	
\newcommand{\Lotimes}[3][R]{#2\otimes^{\mathbf{L}}_{#1}#3}
\newcommand{\Otimes}[3][R]{#2\otimes_{#1}#3}
\renewcommand{\Hom}[3][R]{\operatorname{Hom}_{#1}(#2,#3)}	
\newcommand{\Tor}[4][R]{\operatorname{Tor}^{#1}_{#2}(#3,#4)}

\newcommand{\LL}[2]{\mathbf{L}\Lambda^{\ideal{#1}}(#2)}

\newcommand{\RG}[2]{\mathbf{R}\Gamma_{\ideal{#1}}(#2)}

\newcommand{\width}{\operatorname{width}}

\newcommand{\Comp}[2]{\widehat{#1}^{\ideal{#2}}}
\newcommand{\rad}[1]{\operatorname{rad}(#1)}

\newcommand{\catdfb}{\catd_{\text{b}}^{\text{f}}}
\newcommand{\catdb}{\catd_{\text{b}}}
\newcommand{\catdf}{\catd^{\text{f}}}

\numberwithin{equation}{lem}

\begin{document}

\bibliographystyle{amsplain}

\author{Sean Sather-Wagstaff}

\address{Department of Mathematical Sciences,
Clemson University,
O-110 Martin Hall, Box 340975, Clemson, S.C. 29634
USA}

\email{ssather@clemson.edu}

\urladdr{https://ssather.people.clemson.edu/}

\thanks{
Sean Sather-Wagstaff was supported in part by a grant from the NSA}

\author{Richard Wicklein}

\address{Richard Wicklein, Mathematics and Physics Department, MacMurray College, 447 East College Ave., Jacksonville, IL 62650, USA}

\email{richard.wicklein@mac.edu}

\title
{Adic Finiteness: Bounding Homology and Applications}




\begin{abstract}
We prove a versions of amplitude inequalities of Iversen, Foxby and Iyengar, and Frankild and Sather-Wagstaff
that replace finite generation conditions with adic finiteness conditions. 
As an application, we prove that a local ring $R$ of prime characteristic is regular if and only if 
for some proper ideal $\fb$ the derived local cohomology complex $\RG bR$ has finite flat dimension when viewed through
some positive power of the Frobenius endomorphism.
\end{abstract}

\maketitle

\tableofcontents

\section{Introduction} \label{sec130805a}
Throughout this paper let $R$ and $S$ be 
commutative noetherian rings, let $\fa \subsetneq R$ be a proper ideal of $R$, and let $\Comp{R}{a}$ be the $\fa$-adic completion of $R$.
Let $K=K^R(\x)$ denote the Koszul complex over $R$ on a  generating sequence $\x=x_1,\ldots,x_n$ for $\fa$.
We work in the derived category $\catd(R)$ with objects  the $R$-complexes
indexed homologically
$X=\cdots\to X_i\to X_{i-1}\to\cdots$.
The $i$th shift (or suspension) of $X$ is denoted $\shift^iX$.
We  consider the following full triangulated subcategories of $\catd(R)$.

\

$\catdb(R)$: objects are the complexes $X$ with $\HH_i(X)=0$ for $|i|\gg 0$.

$\catdf(R)$: objects are the complexes $X$ with $\HH_i(X)$ finitely generated for all $i$.

$\catdfb(R):=\catdf(R)\bigcap\catdb(R)$.

\

\noindent 
Isomorphisms in $\catd(R)$ are identified by the symbol $\simeq$.
The appropriately derived functors of $\Hom_R$ and $\otimes_R$ are $\mathbf{R}\!\operatorname{Hom}_R$ and $\otimes^{\mathbf{L}}_R$.
See Section~\ref{sec140109b} for background material and, e.g., \cite{hartshorne:rad,verdier:cd,verdier:1}  for foundations.

\

This work is part 3 of a series of papers exploring notions of support and finiteness of $R$-complexes.
It builds on our previous papers~\cite{sather:afcc, sather:scc}, and it is used in the papers~\cite{sather:afc,sather:asc,sather:elclh}.
It is heavily influenced by Foxby and Iyengar's paper~\cite{foxby:daafuc} and the non-local extension in~\cite{frankild:rrhffd}.

A major point of~\cite{foxby:daafuc} is to prove an amplitude inequality extending
a result of Iversen~\cite{iversen:aifc} to the realm of unbounded complexes.
This implies that, given a local ring homomorphism $R\to S$ with complexes $X\in\catdf(R)$ and $F\in\catdfb(S)$
such that $F\not\simeq 0$ and $\fd_R(F)<\infty$, then one has $X\in\catdb(R)$ if and only if $\Lotimes FX\in\catdb(R)$;
that is, one has $\HH_i(X)=0$ for $|i|\gg 0$ if and only if $\Tor iFX=0$ for $|i|\gg 0$. 
This is extended to the non-local arena in~\cite{frankild:rrhffd}.

In the current paper, we extend these and other results 
to the realm of complexes that do not necessarily have finitely generated homology modules,
but instead have finitely generated Koszul homology modules and restricted support. These are the ``adically finite complexes'',
introduced in~\cite{sather:scc}. For instance,  an $R$-module $M$ is $\fa$-adically finite if it is $\fa$-torsion and has
$\HH_i(\Otimes KM)$ finitely generated for all $i$. 
See Definition~\ref{def120925d} for the general definition. 
In this context, our generalization of the results from the previous paragraph is the following, which is a consequence of
Theorem~\ref{thm151105b} below.

\begin{thm}\label{thm151128c}
Let $\vf\colon R\to S$ be a ring homomorphism such that $\fa S\neq S$, and let $\vf^*$ be the induced map $\spec(S)\to\spec(R)$.
Let $F\in\catdb(S)$ be $\fa S$-adically finite such that  $\fd_R(F)<\infty$  
and $\vf^*(\supp_S(F))\supseteq\VE(\fa)\bigcap\mspec(R)$.
Let $X\in\catd(R)$ be such that  $\supp_R(X)\subseteq\VE(\fa)$ and
$\Lotimes{K}X\in\catdf(R)$.
Then one has $X\in\catdb(R)$ if and only if $\Lotimes FX\in\catdb(R)$.
\end{thm}

The point of this and most of the other results of this paper is that, in the presence of reasonable support conditions,
one can relax homologically finite assumptions to adically finite assumptions.
Sections~\ref{sec151104a}--\ref{sec151211b} contains numerous results akin to Theorem~\ref{thm151128c}, with various derived functors
and finiteness conditions. It should be reiterated that these results are all applied in
our subsequent work, especially in~\cite{sather:afc}.
We also note that many of the results of Section~\ref{sec151104a} are new even when the adically finite condition is replaced with the more
restrictive assumption of being in $\catdfb(R)$. For instance, the next result is a special case of Theorem~\ref{thm151210e}.

\begin{thm}\label{thm151210ew}
Let $F\in\catdfb(R)$ be  such that  $\fd_R(F)<\infty$.
Let $Z\in\catd(R)$ be such that 
$\supp_R(Z)\subseteq\supp_R(F)$.
One has  $Z\in\catdb(R)$ if and only if $\Lotimes FZ\in\catdb(R)$.
\end{thm}

Section~\ref{sec151126a} applies these results to the study of homological properties of local ring homomorphisms.
For instance, the next result, contained in Theorem~\ref{thm151126b}, is a significant extension of~\cite[Theorem~3.3]{foxby:daafuc}.

\begin{thm}\label{thm151126bz}
Let $(R,\m)$ be a local ring of prime characteristic, and let $\vf\colon R\to R$ be the Frobenius endomorphism. 
For an $R$-complex $X$, let ${}^{\vf^t}\!X$ denote the complex $X$ viewed as an $R$-complex by restriction of scalars along
the $n$-fold composition $\vf^t$.
Then the following conditions are equivalent.
\begin{enumerate}[\rm(i)]
\item\label{thm151126bz1} $R$ is regular.
\item\label{thm151126bz2} $\vf^t$ is flat for each integer $t\geq 1$.
\item\label{thm151126bz3} One has $\fd_R({}^{\vf^t}\RG bR)<\infty$ for some integer $t\geq 1$ and some ideal $\fb\subseteq\m$.
\item\label{thm151126bz4} One has $\fd_R({}^{\vf^t}\!F)<\infty$ for some integer $t\geq 1$ for some $\fb$-adically finite $R$-complex 
$F\not\simeq 0$ with $\fd_R(F)<\infty$ for some ideal $\fb\subseteq\m$.
\end{enumerate}
\end{thm}

Here $\RG bR$ is the derived local cohomology complex with respect to $\fb$.
Further applications are contained in the subsequent~\cite{sather:afc}.

It is worth noting that there are a number of substitutes for homological finiteness (e.g., for finite generation of modules) in the literature.
For instance, Simon~\cite{simon:shpcm} considers (complexes of) $\fa$-adically complete modules.
We observe that if $X\in\catd(R)$ is such that $\cosupp_R(X)\subseteq\VE(\fa)$, then $X$ is isomorphic in $\catd(R)$ to a complex of flat complete $R$-modules,
namely, the complex $\Lambda^\fa(F)\simeq\LL aX\simeq X$ where $F\simeq X$ is a semi-flat resolution; see Fact~\ref{cor130528a}.
However, our $\fa$-adically finite complexes are different from this. Indeed, let $(R,\m,k)$ be a local ring of positive depth.
The injective hull $E:=E_R(k)$ is not $\m$-adically separated, so it is not isomorphic in $\catd(R)$ to a complex of complete $R$-modules. 
On the other hand, $E$ is $\m$-adically finite by~\cite[Proposition~7.8(b)]{sather:scc}. 

Another replacement for homological finiteness is Porta, Shaul, and Yekutieli's notion of cohomological cofiniteness, from~\cite{yekutieli:ccc}.
The difference between this notion and ours is discussed in depth in~\cite[Section~6]{sather:elclh}.

\section{Background}\label{sec140109b}

\subsection*{Derived Categories}
We expand our menagerie of categories from the introduction to include the next full triangulated subcategories of $\catd(R)$.

\

$\catd_+(R)$: objects are the complexes $X$ with $\HH_i(X)=0$ for $i\ll 0$.

$\catd_-(R)$: objects are the complexes $X$ with $\HH_i(X)=0$ for $i\gg 0$. 

\

\noindent Doubly ornamented subcategories are intersections, e.g., $\catdf_+(R):=\catdf(R)\bigcap\catd_+(R)$.

\subsection*{Homological Dimensions}
An $R$-complex $F$ is \emph{semi-flat}\footnote{In the literature, semi-flat complexes are sometimes called ``K-flat'' or ``DG-flat''.} 
if it consists of flat $R$-modules and the functor $\Otimes F-$ respects quasiisomorphisms,
that is, if it respects injective quasiisomorphisms (see~\cite[1.2.F]{avramov:hdouc}).
A \emph{semi-flat resolution} of an $R$-complex $X$ is a quasiisomorphism $F\xra\simeq X$ such that $F$ is semi-flat.
An $R$-complex $X$ has \emph{finite flat dimension} if it has a bounded semi-flat resolution;
specifically, we have
$$\fd_R(X)=\inf\{\sup\{i\mid F_i\neq 0\}\mid\text{$F\xra\simeq X$ is a semi-flat resolution}\}.$$
The  projective and injective versions of these notions are defined similarly. 

For the following items, consult~\cite[Section 1]{avramov:hdouc} or~\cite[Chapters 3 and 5]{avramov:dgha}.
Bounded below  complexes of flat modules are semi-flat, 
bounded below  complexes of projective modules are semi-projective, and
bounded above  complexes of injective modules are semi-injective.
Semi-projective $R$-complexes are semi-flat. 
Every $R$-complex admits a semi-projective (hence, semi-flat) resolution  and a semi-injective resolution.

\subsection*{Derived Local (Co)homology}
The next notions go back to Grothendieck~\cite{hartshorne:lc} and Matlis~\cite{matlis:kcd,matlis:hps};
see also~\cite{lipman:lhcs,lipman:llcd}.
Let $\Lambda^{\fa}$ denote the $\fa$-adic completion functor, and let
$\Gamma_{\fa}$ be the $\fa$-torsion functor, i.e.,
for an $R$-module $M$ we have
$$\Lambda^{\fa}(M)=\Comp Ma
\qquad
\qquad
\qquad
\Gamma_{\fa}(M)=\{ x \in M \mid \fa^{n}x=0 \text{ for } n \gg 0\}.$$ 
A module $M$ is \textit{$\fa$-torsion} if $\Gamma_{\fa}(M)=M$.

The associated left and right derived functors (i.e., \emph{derived local homology and cohomology} functors)
are  $\LL a-$ and $\RG a-$.
Specifically, given an $R$-complex $X\in\catd(R)$ and a semi-flat resolution $F\xra\simeq X$ and a 
semi-injective resolution $X\xra\simeq I$, then we have $\LL aX\simeq\Lambda^{\fa}(F)$ and $\RG aX\simeq\Gamma_{\fa}(I)$.

\subsection*{Support and Co-support}
The following notions are due to Foxby~\cite{foxby:bcfm} and Benson, Iyengar, and Krause~\cite{benson:csc}.

\begin{defn}\label{defn130503a}
Let $X\in\catd(R)$.
The \emph{small and large support} and  \emph{small co-support} of $X$ are
\begin{align*}
\operatorname{supp}_R(X)
&=\{\mathfrak{p} \in \operatorname{Spec}(R)\mid \Lotimes{\kappa(\p)}X\not\simeq 0 \} \\
\operatorname{Supp}_R(X)
&=\{\mathfrak{p} \in \operatorname{Spec}(R)\mid \Lotimes{R_{\p}}X\not\simeq 0 \} \\
\cosupp_{R}(X)
&=\{\mathfrak{p} \in \operatorname{Spec}(R)\mid \Rhom{\kappa(\p)}X\not\simeq 0 \} 
\end{align*}
where $\kappa(\p):=R_\p/\p R_\p$.
We have a notion of $\operatorname{Co-supp}_R(X)$, as well, but do not need it here.
\end{defn}

Much of the following is from~\cite{foxby:bcfm} when $X$ and $Y$ are appropriately bounded
and from~\cite{benson:lcstc,benson:csc} in general. We refer to~\cite{sather:scc} as a matter of convenience.

\begin{fact}\label{cor130528a}
Let $X,Y\in\catd(R)$. Then we have $\supp_R(X)=\emptyset$ if and only if $X\simeq 0$ if and only if $\cosupp_R(X)=\emptyset$,
because of~\cite[Fact~3.4 and Proposition~4.7(a)]{sather:scc}.
Also, by~\cite[Propositions~3.12 and~4.10]{sather:scc} we have 
\begin{align*}
\supp_{R}(\Lotimes{X}{Y}) 
&= \supp_R(X)\bigcap\supp_R(Y)\\
\cosupp_{R}(\Rhom{X}{Y}) 
&= \supp_R(X)\bigcap\cosupp_R(Y).
\end{align*}
In addition, we know that $\supp_R(X)\subseteq\VE(\fa)$ if and only if 
$X\simeq\RG aX$ if and only if each homology module $\HH_i(X)$ is $\fa$-torsion,
by~\cite[Proposition~5.4]{sather:scc}
and~\cite[Corollary~4.32]{yekutieli:hct}.
Similarly, we have
$\cosupp_R(X)\subseteq\VE(\fa)$ if and only if $\LL aX\simeq X$,
by~\cite[Propositions 5.9]{sather:scc}.
\end{fact}

\subsection*{Adic Finiteness}
The next fact and definition from~\cite{sather:scc} take their cues from work of 
Hartshorne~\cite{hartshorne:adc},
Kawasaki~\cite{kawasaki:ccma,kawasaki:ccc}, and
Melkersson~\cite{melkersson:mci}.

\begin{fact}[\protect{\cite[Theorem 1.3]{sather:scc}}]
\label{thm130612a}
For $X\in\catd_{\text b}(R)$, the next conditions are equivalent.
\begin{enumerate}[\rm(i)]
\item\label{cor130612a1}
One has $\Lotimes{K^R(\underline{y})}{X}\in\catdfb(R)$  for some (equivalently for every) finite generating sequence $\underline{y}$ of $\fa$.
\item\label{cor130612a2}
One has  $\Lotimes{X}{R/\mathfrak{a}}\in\catd^{\text{f}}(R)$.
\item\label{cor130612a3}
One has  $\Rhom{R/\mathfrak{a}}{X}\in\catd^{\text{f}}(R)$.
\end{enumerate}
\end{fact}

\begin{defn}\label{def120925d}
An $R$-complex $X\in\catdb(R)$ is \emph{$\mathfrak{a}$-adically finite} if it satisfies the equivalent conditions of Fact~\ref{thm130612a} and $\operatorname{supp}_R(X) \subseteq \operatorname{V}(\mathfrak{a})$.
\end{defn}

\begin{ex}\label{ex160206a}
Let $X\in\catdb(R)$ be given.
\begin{enumerate}[(a)]
\item \label{ex160206a1}
If $X\in\catdfb(R)$, then we have $\supp_R(X)=\VE(\fb)$ for some ideal $\fb$, and it follows that $X$ is $\fa$-adically finite
whenever $\fa\subseteq\fb$. (The case $\fa=0$ is from~\cite[Proposition~7.8]{sather:scc}, and the general case follows readily.)
\item \label{ex160206a2}
$K$ and $\RG aR$ are $\fa$-adically finite, by~\cite[Fact~3.4 and Theorem~7.10]{sather:scc}.
\item \label{ex160206a3}
The homology modules of $X$ are artinian if and only if there is an ideal $\fa$ of finite colength (i.e., such that $R/\fa$ is artinian)
such that $X$ is $\fa$-adically finite, by~\cite[Proposition~5.11]{sather:afcc}.
\end{enumerate}
\end{ex}

\begin{fact}\label{fact151129a}
Let $X\in\catd(R)$ be given. It is straightforward to show that one has $\supp_R(X)\subseteq\Supp_R(X)$.
A little more work allows one to show that $\supp_R(X)\subseteq\VE(\fa)$ if and only if $\Supp_R(X)\subseteq\VE(\fa)$;
see~\cite[Proposition~3.15]{sather:scc}. 
If $X$ is $\fa$-adically finite, then one has more, by~\cite[Proposition~7.11]{sather:scc}: the set $\supp_R(X)=\Supp_R(X)$
is Zariski-closed in $\spec(R)$.
\end{fact}

\subsection*{Bookkeeping}
We use some convenient accounting tools, due in this generality to Foxby and Iyengar~\cite{foxby:ibcahtm, foxby:daafuc}. 

\begin{defn}\label{defn151112a}
The
\emph{supremum}, \emph{infimum},  \emph{amplitude}, \emph{$\fa$-depth}, and \emph{$\fa$-width} of an $R$-complex $Z$ are
\begin{align*}
\sup(Z)&=\sup\{ i\in\bbz\mid\HH_i(Z)\neq 0\}\\
\inf(Z)&=\inf\{ i\in\bbz\mid\HH_i(Z)\neq 0\}\\
\amp(Z)&=\sup(Z)-\inf(Z)\\
\depth_{\fa}(Z)&=-\sup(\Rhom{R/\fa}Z)\\
\width_{\fa}(Z)&=\inf(\Lotimes{(R/\fa)}Z)
\end{align*}
with the conventions $\sup\emptyset=-\infty$ and $\inf\emptyset=\infty$.
When $(R,\m)$ is local, one sets $\depth_R(Z):=\depth_{\m}(Z)$ and $\width_R(Z):=\width_{\m}(Z)$.
\end{defn}

In the following fact, items~\eqref{disc151112a1}--\eqref{disc151112a4} are by definition, 
and item~\eqref{disc151112a5} is from~\cite[Theorems~2.4.F and~P]{avramov:hdouc} and~\cite[Lemma~2.1]{foxby:ibcahtm}.

\begin{fact}\label{disc151112a}
Let $Y,Z\in\catd(R)$.
\begin{enumerate}[(a)]
\item\label{disc151112a1}
One has $\sup(Z)<\infty$ if and only if $Z\in\catd_-(R)$.
Also, one has $\inf(Z)>-\infty$ if and only if $Z\in\catd_+(R)$,
and one has $\amp(Z)<\infty$ if and only if $Z\in\catdb(R)$.
\item\label{disc151112a4}
The following conditions are equivalent: \\
(i) $Z\simeq 0$,
(ii) $\sup(Z)=-\infty$,
(iii) $\inf(Z)=\infty$, and
(iv) $\amp(Z)=-\infty$.
\item\label{disc151112a5}
There are inequalities
\begin{gather*}
\inf(Y)+\inf(Z)\leq\inf(\Lotimes YZ)\\
\sup(\Lotimes YZ)\leq\sup(Y)+\fd_R(Z)\\
\inf(Z)-\pd_R(Y)\leq\inf(\Rhom YZ)\\
\sup(\Rhom YZ)\leq\sup(Z)-\inf(Y)
\end{gather*}
\end{enumerate}
\end{fact}

\subsection*{Koszul Complexes}
We refer to parts of the following as the ``self-dual nature'' of the Koszul complex.

\begin{fact}\label{disc151211a}
Let $\y=y_1,\ldots,y_m\in R$, and consider the Koszul complex $L:=K^R(\y)$.
Then we have isomorphisms in $\catd(R)$
\begin{gather*}
L\simeq\shift^n\Rhom LR\\
\Lotimes LX\simeq\shift^n\Lotimes{\Rhom LR}X\simeq\shift^n\Rhom LX\\
\Rhom{\Rhom LX}Y
\simeq\Lotimes L{\Rhom XY}
\simeq\Rhom X{\Lotimes LY}.
\end{gather*}
See, e.g., \cite[Remark~2.2]{sather:afcc}. 
\end{fact}

The following facts are from~\cite[Lemmas~3.1--3.2]{sather:afcc}.

\begin{fact}\label{lem150604a1}
Let  $Z\in\catd(R)$, let $\y=y_1,\ldots,y_m\in\fa$, and set
$L:=K^R(\y)$ and $\fb=(\y)R$.
Assume that $\supp_R(Z)\subseteq\VE(\fa)$, e.g., that each homology module $\HH_i(Z)$ is annihilated by a power of $\fa$.
\begin{enumerate}[\rm(a)]
\item\label{lem150604a1z}
There are (in)equalities
\begin{align*}
\inf(\Lotimes LZ)&\leq m+\inf(Z)&
\sup(\Lotimes LZ)&=m+\sup(Z)\\
\amp(Z)&\leq\amp(\Lotimes LZ)&
\depth_{\fb}(Z)&=-\sup(Z).
\end{align*}
\item\label{lem150604a1a}
For each $*\in\{+,-,\text{b}\}$, one has $\Lotimes LZ\in\catd_*(R)$ if and only if $Z\in\catd_*(R)$.
\end{enumerate}
\end{fact}

Note that some items in the next result use $L$, while others use $K$. 

\begin{fact}\label{lem150604a2}
Let  $Z\in\catd(R)$, let $\y=y_1,\ldots,y_m\in\fa$, and set
$L:=K^R(\y)$ and $\fb:=(\y)R\subseteq\fa$. 
Assume that $\cosupp_R(Z)\subseteq\VE(\fa)$, e.g., that each homology module $\HH_i(Z)$ is annihilated by a power of $\fa$.
\begin{enumerate}[\rm(a)]
\item\label{lem150604a2z}
There are (in)equalities
\begin{gather*}
\width_\fb(Z)=\inf(Z)=\inf(\Lotimes LZ)\\
\sup(Z)-n\leq\sup(\Lotimes KZ) \\
\amp(Z)-n\leq\amp(\Lotimes KZ).
\end{gather*}
\item\label{lem150604a2b}
For each $*\in\{+,-,\text{b}\}$, one has $\Lotimes KZ\in\catd_*(R)$ if and only if $Z\in\catd_*(R)$.
\end{enumerate}
\end{fact}

\section{Bounding Homology I: Maximal Support}\label{sec151104a}
In this section, we extend results from~\cite{foxby:ibcahtm, foxby:daafuc, frankild:rrhffd}
to the adically finite arena.

\subsection*{Bounding by Ext}

We begin  with two improvements of a result from~\cite{foxby:ibcahtm} to the adically finite realm.
Note that each of these results assumes enough boundedness on $X$ and $Y$ to imply that $\Rhom XY\in\catd_-(R)$,
in contrast with  other results below.
The point of our results is to give specific ranges for non-vanishing homology. For instance, Proposition~\ref{prop151126a}, 
in conjunction with Fact~\ref{disc151112a}\eqref{disc151112a5}, 
implies that
$\HH_i(\Rhom XY)\neq 0$ for some $i\in\{\sup(Y)-\sup(X),\ldots,\sup(Y)-\inf(X)\}$, when $Y\not\simeq 0$.\footnote{The conditions
$Y\in\catd_-(R)$ and $Y\not\simeq 0$ imply that $\sup(Y)\in\bbz$. Note that we do not need to make a similar assumption for $X$,
since the condition $\fa\neq R$ implies that $\supp_R(X)=\VE(\fa)\neq 0$, so we have $X\not\simeq 0$, and thus
$\sup(X),\inf(X)\in\bbz$.}
Also, as far as we know, the non-trivial implication in part Proposition~\ref{prop151126a}~\eqref{prop151126a2} does
not follow just from support considerations; and similarly for other results below.

\begin{prop}\label{prop151126a}
Let $X\in\catdb(R)$ be $\fa$-adically finite, and let $Y\in\catd_-(R)$ be such that $\supp_R(Y)\subseteq\VE(\fa)=\supp_R(X)$.
\begin{enumerate}[\rm(a)]
\item\label{prop151126a1}
Then there is an inequality
\begin{align*}
\sup(Y)-\sup(X)-n\leq\sup(\Rhom XY).
\end{align*}
\item\label{prop151126a2}
One has $Y\simeq 0$ if and only if $\Rhom XY\simeq 0$. 
\end{enumerate}
\end{prop}

\begin{proof}
\eqref{prop151126a1}
Claim: If $X\in\catdfb(R)$, then
$\sup(Y)-\sup(X)\leq\sup(\Rhom XY)$.
Indeed, the condition $\supp_R(Y)\subseteq\VE(\fa)$ implies that $\HH_i(Y)$ is $\fa$-torsion for all $i$,
so
$$\Supp_R(\HH_i(Y))\subseteq\VE(\fa)=\supp_R(X)=\Supp_R(X)$$
by Fact~\ref{fact151129a}.
Thus, the desired inequality is  from~\cite[Proposition~2.2]{foxby:ibcahtm}.\footnote{Note that complexes in~\cite{foxby:ibcahtm}
are indexed cohomologically, so one has to translate~\cite[Proposition~2.2]{foxby:ibcahtm} carefully.}

Now we prove the inequality $\sup(Y)-\sup(X)-n\leq\sup(\Rhom XY)$ in general.
By assumption, we have $\Lotimes KX\in\catdfb(R)$, so the above claim explains the second step in the next display.
\begin{align*}
\sup(Y)-\sup(X)-n
&=\sup(Y)-\sup(\Lotimes KX) \\
&\leq\sup(\Rhom{\Lotimes KX}{Y})\\
&=\sup(\Rhom{K}{\Rhom XY})\\
&\leq\sup(\Rhom XY)-\inf K \\
&=\sup(\Rhom XY)
\end{align*}
The first step is from Fact~\ref{lem150604a1}\eqref{lem150604a1z} because $\supp_R(X)\subseteq\VE(\fa)$.
The third step is by  adjointness 
$\Rhom{\Lotimes KX}{Y}\simeq\Rhom{K}{\Rhom XY}$,
and the fourth step is from Fact~\ref{disc151112a}\eqref{disc151112a5}.
The last step is due to the equality $\inf(K)=0$.

\eqref{prop151126a2}
The forward implication is standard. For the converse,
assume that we have $\Rhom XY\simeq 0$. It follows that 
$$\sup(Y)-\sup(X)-n\leq\sup(\Rhom XY)=-\infty.$$
Our assumptions guarantee that $\sup(X)\in\bbz$, so we conclude that $\sup(Y)=-\infty$,
thus $Y\simeq 0$.
\end{proof}

\begin{prop}\label{prop151126az}
Let $X\in\catdb(R)$ be $\fa$-adically finite, and let $Y\in\catd_-(R)$ be such that $\cosupp_R(Y)\subseteq\VE(\fa)=\supp_R(X)$.
Then there is an inequality
\begin{align*}
\sup(Y)-\sup(X)-3n\leq\sup(\Rhom XY).
\end{align*}
\end{prop}

\begin{proof}
The desired inequality follows from the next sequence.
\begin{align*}
\sup(\Rhom XY)
&\geq\sup(\Lotimes K{\Rhom XY})-n \\
&\geq\sup(\Rhom X{\Lotimes KY})-n \\
&\geq\sup(\Lotimes KY)-\sup(X)-2n \\
&\geq\sup(Y)-\sup(X)-3n
\end{align*}
The first inequality is from Fact~\ref{disc151112a}\eqref{disc151112a5},
and the second inequality is from  tensor-evaluation~\ref{disc151211a}.
Fact~\ref{cor130528a} implies that $\supp_R(\Lotimes KY)\subseteq\supp_R(K)\subseteq\VE(\fa)$, so the third inequality is from Proposition~\ref{prop151126a}\eqref{prop151126a1}.
The fourth inequality is from Fact~\ref{lem150604a2}\eqref{lem150604a2z}.
\end{proof}

Our next results are more along the lines of Theorem~\ref{thm151128c} from
the introduction: in the presence of support and (adic) finiteness assumptions,
boundedness or triviality of $\Rhom PY$ implies the same for $Y$. 
See Lemma~\ref{thm151210cw} for an improvement on the following.

\begin{lem}\label{thm151112aw}
Let $P\in\catdfb(R)$ be such that $P\not\simeq 0$ and $\pd_R(P)<\infty$.
Let $Y\in\catd(R)$ be such that 
$\supp_R(Y)\subseteq\supp_R(P)$.
\begin{enumerate}[\rm(a)]
\item\label{thm151112aw1}
One has $\sup(Y)-\sup(P)\leq\sup(\Rhom PY)$.
\item\label{thm151112aw2}
One has $Y\in\catd_-(R)$ if and only if $\Rhom PY\in\catd_-(R)$.
\item\label{thm151112aw3}
One has $Y\simeq 0$ if and only if $\Rhom PY\simeq 0$.
\end{enumerate}
\end{lem}

\begin{proof}
Note first that the conditions $0\not\simeq P\in\catdfb(R)$ imply that $\supp_R(P)=\VE(\fa)$ for some $\fa\neq R$.

\eqref{thm151112aw1}
Because of the assumption $\supp_R(Y)\subseteq\VE(\fa)$, the first step in the next sequence is from
Fact~\ref{lem150604a1}\eqref{lem150604a1z}.
\begin{align*}
-\sup(Y)
&=\depth_{\fa}(Y) \\
&=\inf\{\depth_{R_{\p}}(Y_{\p})\mid\p\in\VE(\fa)\}\\
&=\inf\{\depth_{R_{\p}}(\Rhom[R_{\p}]{P_{\p}}{Y_{\p}})-\width_{R_{\p}}(P_{\p})\mid\p\in\VE(\fa)\}\\
&=\inf\{\depth_{R_{\p}}(\Rhom{P}{Y}_{\p})-\inf(P_{\p})\mid\p\in\VE(\fa)\}\\
&\geq-\sup(\Rhom PY)-\sup(P).
\end{align*}
The second and third steps are from~\cite[Propositions 2.10 and 4.6]{foxby:daafuc};
note that the third step uses  the assumption $\supp_R(P)=\VE(\fa)$.
For the fourth step, we use the conditions $P \in\catdfb(R)$ and $\pd_R(P)<\infty$
to conclude that $\Rhom[R_{\p}]{P_{\p}}{Y_{\p}}\simeq\Rhom{P}{Y}_{\p}$, e.g., by~\cite[Proposition~2.1(iii)]{christensen:apac};
for the equality $\width_{R_{\p}}(P_{\p})=\inf(P_{\p})$, see, e.g., \cite[(1.2.1)]{christensen:rhdacm}.
To explain the last step in this display, we have the following, 
by definition and Fact~\ref{disc151112a}\eqref{disc151112a5}
\begin{align*}
\depth_{R_{\p}}(\Rhom{P}{Y}_{\p})
&=-\sup(\Rhom[R_{\p}]{\kappa(\p)}{\Rhom{P}{Y}_{\p}})\\
&\geq-\sup(\Rhom{P}{Y}_{\p})\\
&\geq-\sup(\Rhom PY)
\end{align*}
with the next sequence
$$\inf(P_{\p})\leq\sup(P_{\p})\leq\sup(P)$$
where the assumption $\supp_R(P)=\VE(\fa)$ is used.

\eqref{thm151112aw2}--\eqref{thm151112aw3}
Our assumptions on $P$ imply that  $\inf(P),\sup(P)\in\bbz$.
Thus, the desired conclusions follow directly from part~\eqref{thm151112a1}
because of Fact~\ref{disc151112a}\eqref{disc151112a1}.\footnote{The following alternate proof of part~\eqref{thm151112aw3} is worth noting.
By Hom-evaluation~\cite[Lemma~4.4(I)]{avramov:hdouc}, the assumptions on $P$ provide an isomorphism
$\Rhom PY\simeq\Lotimes{P^*}{Y}$ where $P^*=\Rhom PR$ satisfies $\supp_R(P^*)=\supp_R(P)$, as in the proof of Corollary~\ref{cor151114b}
below.
Now apply Fact~\ref{cor130528a}.}
\end{proof}

See Theorem~\ref{thm151210c} for an improvement on our next result.
It should be noted that the condition $\supp_R(Y)\subseteq\VE(\fa)$ is a bit strange to us:
given the co-support formula in Fact~\ref{cor130528a}, it would seem more natural to assume 
$\cosupp_R(Y)\subseteq\VE(\fa)$; see, however, Theorem~\ref{thm151210a}.

\begin{thm}\label{thm151112a}
Let $P\in\catdb(R)$ be $\fa$-adically finite such that  $\pd_R(P)<\infty$ and $\supp_R(P)=\VE(\fa)$.
Let $Y\in\catd(R)$ be such that 
$\supp_R(Y)\subseteq\VE(\fa)$.
\begin{enumerate}[\rm(a)]
\item\label{thm151112a1}
One has $\sup(Y)-\sup(P)-n\leq\sup(\Rhom PY)$.
\item\label{thm151112a2}
One has $Y\in\catd_-(R)$ if and only if $\Rhom PY\in\catd_-(R)$.
\item\label{thm151112a3}
One has $Y\simeq 0$ if and only if $\Rhom PY\simeq 0$.
\end{enumerate}
\end{thm}

\begin{proof}
\eqref{thm151112a1}
Set $K^*:=\Rhom KR\simeq\shift^{-n}K$, so we have 
$$\pd_R(K^*)=\pd_R(K)-n=0.$$
This explains the first step in the next display.
\begin{align*}
\sup(\Rhom PY)
&=\pd_R(K^*)+\sup(\Rhom PY)\\
&\geq\sup(\Lotimes{K^*}{\Rhom PY})\\
&=\sup(\Rhom {\Lotimes KP}{Y})\\
&\geq\sup(Y)-\sup(\Lotimes KP)\\
&\geq\sup(Y)-n-\sup(P)
\end{align*}
The second and last steps are by Fact~\ref{disc151112a}\eqref{disc151112a5}.
The fourth step is from Lemma~\ref{thm151112aw}\eqref{thm151112aw1}, applied to the complex $\Lotimes KP\in\catdfb(R)$
which has $\supp_R(\Lotimes KP)=\VE(\fa)$. The third step is from the  evaluation and adjunction isomorphisms
\begin{align*}
\Lotimes{K^*}{\Rhom PY}
&\simeq\Rhom P{\Lotimes{K^*}{Y}}\\
&\simeq\Rhom P{\Rhom{K}{Y}}\\
&\simeq\Rhom {\Lotimes KP}{Y}
\end{align*}
see Fact~\ref{disc151211a}. 

\eqref{thm151112a2}--\eqref{thm151112a3}
Note that 
our assumptions on $P$ imply that $0\not\simeq P\in\catdb(R)$, so we have $\inf(P),\sup(P)\in\bbz$.
Thus, the desired conclusions follow directly from part~\eqref{thm151112a1}
because of Fact~\ref{disc151112a}\eqref{disc151112a1}.
\end{proof}

\begin{disc}\label{disc151106a}
As in~\cite{sather:scc}, Theorem~\ref{thm151112a} has the following consequences. 
Let $P\in\catdb(R)$ be $\fa$-adically finite with  $\pd_R(P)<\infty$ and $\VE(\fa)=\supp_R(P)$.
Let $f\colon X\to Y$ be a morphism in $\catd(R)$ with 
$\supp_R(X),\supp_R(Y)\subseteq\VE(\fa)$.
Then $f$ is an isomorphism if and only if $\Rhom Pf$ is an isomorphism. 
We resist the urge to document every possible variation on this theme, here and elsewhere in the paper.

Each result like Theorem~\ref{thm151112a} has a version for $P\in\catdfb(R)$ Like Lemma~\ref{thm151112aw}.
While most of these results are new, we resist the urge to document them all, for the sake of brevity.
\end{disc}

See Theorem~\ref{thm151210d} for an improvement of the next result.

\begin{thm}\label{thm151210a}
Let $P\in\catdb(R)$ be $\fa$-adically finite such that  $\pd_R(P)<\infty$ and $\supp_R(P)=\VE(\fa)$.
Let $Y\in\catd(R)$ be such that 
$\cosupp_R(Y)\subseteq\VE(\fa)$.
\begin{enumerate}[\rm(a)]
\item\label{thm151210a1}
One has $\sup(Y)-\sup(P)-3n\leq\sup(\Rhom PY)$.
\item\label{thm151210a2}
One has $Y\in\catd_-(R)$ if and only if $\Rhom PY\in\catd_-(R)$.
\end{enumerate}
\end{thm}

\begin{proof}
It suffices to prove part~\eqref{thm151210a1}.
To this end, the first step in the next sequence is from Fact~\ref{disc151112a}\eqref{disc151112a5}.
\begin{align*}
\sup(\Rhom PY)
&\geq\sup(\Lotimes K{\Rhom PY})-n \\
&=\sup(\Rhom P{\Lotimes KY})-n \\
&\geq\sup(\Lotimes KY)-\sup(P)-2n \\
&\geq\sup(Y)-\sup(P)-3n 
\end{align*}
The second step  is by tensor-evaluation~\ref{disc151211a},
and the third step follows from Theorem~\ref{thm151112a}\eqref{thm151112a1} since we have
$\supp_R(\Lotimes KY)\subseteq\supp_R(K)=\VE(\fa)$.
The last step is from Fact~\ref{lem150604a2}\eqref{lem150604a2z}.
\end{proof}

\subsection*{Bounding by  Tor}
Here is a version of Propositions~\ref{prop151126a}--\ref{prop151126az} for Tor.
As with those results, the point is to guarantee the existence of non-zero Tor-modules,
when one assumes enough boundedness on $X$ and $Y$ to guarantee that $\Lotimes XY\in\catd_+(R)$.

\begin{prop}\label{prop151126b}
Let $X\in\catdb(R)$ be $\fa$-adically finite with $\supp_R(X)=\VE(\fa)$.
Let $Y\in\catd_+(R)$ be such that either $\supp_R(Y)\subseteq\VE(\fa)$ or $\cosupp_R(Y)\subseteq\VE(\fa)$.
\begin{enumerate}[\rm(a)]
\item \label{prop151126b1}
Then there is an inequality
\begin{align*}
\inf(\Lotimes XY)\leq\sup(X)+\inf(Y)+2n.
\end{align*}
\item \label{prop151126b2}
One has $Y\simeq 0$ if and only if $\Lotimes XY\simeq 0$. 
\end{enumerate}
\end{prop}

\begin{proof}
\eqref{prop151126b1}
Let $E$ be a faithfully injective $R$-module and set $(-)^\vee:=\Rhom -E$.

Assume first that $\supp_R(Y)\subseteq\VE(\fa)$.
The first step in the following display is from Fact~\ref{disc151112a}\eqref{disc151112a5} since $\inf(K)=0$.
\begin{align*}
\inf(\Lotimes XY)
&\leq\inf(\Lotimes X{\Lotimes YK}) \\
&=-\sup((\Lotimes X{\Lotimes YK})^\vee) \\
&=-\sup(\Rhom X{(\Lotimes YK)^\vee})\\
&\leq-\sup((\Lotimes YK)^\vee)+\sup(X)+n \\
&=\inf(\Lotimes YK)+\sup(X)+n \\
&\leq\inf(Y)+\sup(X)+2n 
\end{align*}
The second and fifth steps are from the faithful injectivity of $E$,
and the third step is from adjointness.
For the fourth step, note that $\fa$ annihilates the homology of $\Lotimes YK$, 
so it also annihilates the homology of $(\Lotimes YK)^\vee$;
it follows that $\supp_R((\Lotimes YK)^\vee)\subseteq\VE(\fa)$.
Since we also have $(\Lotimes YK)^\vee\in\catd_-(R)$, the fourth step follows from Proposition~\ref{prop151126a}\eqref{prop151126a1}.
The sixth step is from Fact~\ref{lem150604a1}\eqref{lem150604a1z}.

Assume next that $\cosupp_R(Y)\subseteq\VE(\fa)$.
The first step in the next display is from Fact~\ref{disc151112a}\eqref{disc151112a5}.
\begin{align*}
\inf(\Lotimes XY)
&\leq\inf(\Lotimes X{\Lotimes YK}) \\
&\leq\inf(\Lotimes YK)+\sup(X)+2n \\
&=\inf(Y)+\sup(X)+2n 
\end{align*}
Since $\supp_R(\Lotimes KY)\subseteq\supp_R(K)\subseteq\VE(\fa)$, the second step follows from the previous paragraph.
The third step is from Fact~\ref{lem150604a2}\eqref{lem150604a2z}.

\eqref{prop151126b2} This follows from part~\eqref{prop151126b1}, using Fact~\ref{disc151112a}\eqref{disc151112a4}.
\end{proof}

We continue with more results along the lines of Theorem~\ref{thm151128c}
See Theorem~\ref{thm151210f} below for an improvement.

\begin{thm}\label{thm151112b}
Let $F\in\catdb(R)$ be $\fa$-adically finite such that  $\fd_R(F)<\infty$ and $\VE(\fa)=\supp_R(F)$.
Let $Z\in\catd(R)$ be such that 
$\cosupp_R(Z)\subseteq\VE(\fa)$.
\begin{enumerate}[\rm(a)]
\item\label{thm151112b1}
One has $\inf(\Lotimes FZ)\leq\inf(Z)+\sup(F)+n$.
\item\label{thm151112b2}
One has $Z\in\catd_+(R)$ if and only if $\Lotimes FZ\in\catd_+(R)$.
\item\label{thm151112b3}
One has $Z\simeq 0$ if and only if $\Lotimes FZ\simeq 0$.
\end{enumerate}
\end{thm}

\begin{proof}
\eqref{thm151112b1}
We first prove the special case where $\fa\HH(Z)=0$.
Let $E$ be a faithfully injective $R$-module, and set $(-)^\vee:=\Rhom -E$.
The assumption $\fa\HH(Z)=0$ implies that $\fa\HH(Z^\vee)=0$ since $E$ is injective. 
In particular, we have $\supp_R(Z^\vee)\subseteq\VE(\fa)$ by Fact~\ref{cor130528a}.
The fact that $E$ is faithfully injective explains the first and last steps in the next sequence.
\begin{align*}
\inf(\Lotimes FZ)
&=-\sup((\Lotimes FZ)^\vee)\\
&=-\sup(\Rhom F{Z^\vee}) \\
&\leq n+\sup(F)-\sup(Z^\vee)\\
&=n+\sup(F)+\inf(Z)
\end{align*}
The second step is from Hom-tensor adjointness.
The third step is from Theorem~\ref{thm151112a}\eqref{thm151112a1},
since~\cite[Theorem~6.1]{sather:afcc} implies $\pd_R(F)<\infty$.

Now we deal with the general case. 
The complex $\Lotimes ZK$ satisfies $\fa \HH(\Lotimes ZK)=0$. 
Thus, the previous paragraph explains the second step in the next sequence.
\begin{align*}
\inf(\Lotimes FZ)
&\leq\inf(\Lotimes F{\Lotimes ZK})\\
&\leq n+\sup(F)+\inf(\Lotimes ZK) \\
&=n+\sup(F)+\inf(Z)
\end{align*}
The first step is from Fact~\ref{disc151112a}\eqref{disc151112a5}, and the third step is from 
Fact~\ref{lem150604a2}\eqref{lem150604a2z}, wherein our co-support assumption is used. 

\eqref{thm151112b2}--\eqref{thm151112b3}
These follow from part~\eqref{thm151112b1} and Fact~\ref{disc151112a}\eqref{disc151112a5}.
\end{proof}

See Theorem~\ref{thm151210e} for an improvement of the next result.

\begin{thm}\label{thm151210b}
Let $F\in\catdb(R)$ be $\fa$-adically finite such that  $\fd_R(F)<\infty$ and $\VE(\fa)=\supp_R(F)$.
Let $Z\in\catd(R)$ be such that 
$\supp_R(Z)\subseteq\VE(\fa)$.
\begin{enumerate}[\rm(a)]
\item\label{thm151210b1}
One has $\inf(\Lotimes FZ)\leq\inf(Z)+\sup(F)+2n$.
\item\label{thm151210b2}
One has $Z\in\catd_+(R)$ if and only if $\Lotimes FZ\in\catd_+(R)$.
\end{enumerate}
\end{thm}

\begin{proof}
It suffices to prove part~\eqref{thm151210b1}

The complex $\Lotimes ZK$ satisfies $\fa \HH(\Lotimes ZK)=0$. 
Thus, Theorem~\ref{thm151112b}\eqref{thm151112b1} explains the second step in the next sequence.
\begin{align*}
\inf(\Lotimes FZ)
&\leq\inf(\Lotimes F{\Lotimes ZK})\\
&\leq n+\sup(F)+\inf(\Lotimes ZK) \\
&\leq2n+\sup(F)+\inf(Z)
\end{align*}
The first and third steps are from Facts~\ref{disc151112a}\eqref{disc151112a5} and~\ref{lem150604a1}\eqref{lem150604a1z}.
\end{proof}

\subsection*{Ext Revisited}

Each of the preceding results of this section deals with only one invariant, either $\sup(\Rhom PY)$ or $\inf(\Lotimes FZ)$. 
We now show how to use these results to bootstrap our way to other invariants, e.g., $\inf(\Rhom PY)$ and $\sup(\Lotimes FZ)$,
beginning with an improvement of Theorem~\ref{thm151112a}.

\begin{lem}\label{thm151210cw}
Let $P\in\catdb(R)$ be $\fa$-adically finite such that $p\not\simeq 0$ and $\pd_R(P)<\infty$.
Let $Y\in\catd(R)$ be such that 
$\supp_R(Y)\subseteq\supp_R(P)$.
\begin{enumerate}[\rm(a)]
\item\label{thm151210cw1}
There are inequalities
\begin{gather*}
\inf(\Rhom PY)\leq\inf(Y)-\inf(P)+2n\\
\sup(Y)-\sup(P)\leq\sup(\Rhom PY)\\
\amp(Y)-\amp(P)-2n\leq\amp(\Rhom PY).
\end{gather*}
\item\label{thm151210cw2}
For  $*\in\{+,-,\text{b}\}$, one has  $Y\in\catd_*(R)$ if and only if $\Rhom PY\in\catd_*(R)$.
\end{enumerate}
\end{lem}

\begin{proof}
In light of Fact~\ref{disc151112a} and
Lemma~\ref{thm151112aw}\eqref{thm151112aw1}, it suffices to prove the first inequality in part~\eqref{thm151210cw1}.
Set $P^*:=\Rhom PR$ which is in $\catdfb(R)$ and has $\pd_R(P^*)<\infty$. 
Also, from Fact~\ref{disc151112a}\eqref{disc151112a5}, we have $\sup(P^*)\leq-\inf(P)$, and thus the last step in the next sequence.
\begin{align*}
\inf(Y)
&\geq\inf(\Lotimes YK)-n \\
&\geq\inf(\Lotimes{P^*}{\Lotimes YK})-\sup(P^*)-2n \\
&\geq\inf(\Lotimes{P^*}{Y})-\sup(P^*)-2n \\
&=\inf(\Rhom PY)-\sup(P^*)-2n \\
&\geq \inf(\Rhom PY)+\inf(P)-2n
\end{align*}
The first step is from Fact~\ref{lem150604a1}\eqref{lem150604a1z}.
The second step is by Theorem~\ref{thm151112b}\eqref{thm151112b1}, since $\fa$ annihilates $\HH(\Lotimes YK)$.
The third step is from Fact~\ref{disc151112a}\eqref{disc151112a5},
and the fourth step is from the isomorphisms
$\Lotimes{P^*}{Y}\simeq \Rhom {P^{**}}Y\simeq \Rhom PY$; see~\cite[Lemma~4.4(I)]{avramov:hdouc}.
\end{proof}

\begin{thm}\label{thm151210c}
Let $P\in\catdb(R)$ be $\fa$-adically finite such that  $\pd_R(P)<\infty$ and $\supp_R(P)=\VE(\fa)$.
Let $Y\in\catd(R)$ be such that 
$\supp_R(Y)\subseteq\VE(\fa)$.
\begin{enumerate}[\rm(a)]
\item\label{thm151210c1}
There are inequalities
\begin{gather*}
\inf(\Rhom PY)\leq\inf(Y)-\inf(P)+3n\\
\sup(Y)-\sup(P)-n\leq\sup(\Rhom PY)\\
\amp(Y)-\amp(P)-4n\leq\amp(\Rhom PY).
\end{gather*}
\item\label{thm151210c2}
For  $*\in\{+,-,\text{b}\}$, one has  $Y\in\catd_*(R)$ if and only if $\Rhom PY\in\catd_*(R)$.
\end{enumerate}
\end{thm}

\begin{proof}
In light of Fact~\ref{disc151112a} and
Theorem~\ref{thm151112a}\eqref{thm151112a1}, it suffices to prove the first inequality in part~\eqref{thm151210c1}.
To this end, we apply Lemma~\ref{thm151210cw}\eqref{thm151210cw1} to the complex $\Rhom KP\in\catdfb(R)$.
For this, note that $\pd_R(\Rhom KP)<\infty$
and $\Rhom KP\simeq\shift^{-n}\Lotimes KP$ by Fact~\ref{disc151211a}.
It follows by assumption that 
$$\supp_R(\Rhom KP)=\supp_R(\Lotimes KP)=\supp_R(K)\bigcap\supp_R(P)=\VE(\fa).$$ 
Thus, Lemma~\ref{thm151210cw}\eqref{thm151210cw1} explains the first step in the next display.
\begin{align*}
\inf(Y)
&\geq\inf(\Rhom{\Rhom KP}Y)+\inf(\Rhom KP)-2n \\
&=\inf(\Lotimes K{\Rhom PY})+\inf(\shift^{-n}\Lotimes KP)-2n\\
&=\inf(\Lotimes K{\Rhom PY})+\inf(\Lotimes KP)-3n\\
&\geq\inf(\Rhom PY)+\inf(P)-3n
\end{align*}
The second step is from  Hom-evaluation~\ref{disc151211a}
and the isomorphism $\Rhom KP\simeq \shift^{-n}\Lotimes KP$ noted above. The third
step is straightforward, and the fourth one is from Fact~\ref{disc151112a}\eqref{disc151112a5}.
\end{proof}

The next result follows directly from
Theorems~\ref{thm151112a}
and~\ref{thm151210c} in the special case $F=\RG aR$.
In it we use the notation $\operatorname{cd}_\fa(R):=-\inf(\RG aR)$, which is the \emph{cohomological dimension of $R$ with respect to $\fa$}
and the standard interpretation $\depth_\fa(R)=-\sup(\RG aR)$ of depth in terms of local cohomology.

\begin{cor}\label{cor151213b}
Let $Y\in\catd(R)$ be such that 
$\supp_R(Y)\subseteq\VE(\fa)$.
\begin{enumerate}[\rm(a)]
\item\label{cor151213b1}
There are inequalities
\begin{gather*}
\inf(\LL aY)\leq\inf(Y)+\operatorname{cd}_\fa(R)+3n\\
\sup(Y)+\depth_\fa(R)-n\leq\sup(\LL aY)\\
\amp(Y)+\depth_\fa(R)-\operatorname{cd}_\fa(R)-4n\leq\amp(\LL aY).
\end{gather*}
\item\label{cor151213b2}
For  $*\in\{+,-,\text{b}\}$, one has  $Y\in\catd_*(R)$ if and only if $\LL aY\in\catd_*(R)$.
\item\label{cor151213b3}
One has $Y\simeq 0$ if and only if $\LL aY\simeq 0$.
\end{enumerate}
\end{cor}

Next, we improve on Theorem~\ref{thm151210a}.

\begin{thm}\label{thm151210d}
Let $P\in\catdb(R)$ be $\fa$-adically finite such that  $\pd_R(P)<\infty$ and $\supp_R(P)=\VE(\fa)$.
Let $Y\in\catd(R)$ be such that 
$\cosupp_R(Y)\subseteq\VE(\fa)$.
\begin{enumerate}[\rm(a)]
\item\label{thm151210d1}
There are inequalities
\begin{gather*}
\inf(\Rhom PY)\leq\inf(Y)-\inf(P)+3n\\
\sup(Y)-\sup(P)-3n\leq\sup(\Rhom PY)\\
\amp(Y)-\amp(P)-6n\leq\amp(\Rhom PY).
\end{gather*}
\item\label{thm151210d2}
For  $*\in\{+,-,\text{b}\}$, one has  $Y\in\catd_*(R)$ if and only if $\Rhom PY\in\catd_*(R)$.
\end{enumerate}
\end{thm}

\begin{proof}
Because of Theorem~\ref{thm151210a}\eqref{thm151210a1}, we need only prove the first inequality of part~\eqref{thm151210d1}.
Since we have $\cosupp_R(\Rhom PY)\subseteq\supp_R(P)=\VE(\fa)$, the first step in the next sequence is from Fact~\ref{lem150604a2}\eqref{lem150604a2z}.
\begin{align*}
\inf(\Rhom PY)
&=\inf(\Lotimes K{\Rhom PY}) \\
&=\inf(\Rhom P{\Lotimes KY}) \\
&\leq\inf(\Lotimes KY)-\inf(P)+3n \\
&=\inf(Y)-\inf(P)+3n 
\end{align*}
The second step  is by tensor-evaluation~\ref{disc151211a},
and the third step follows from Theorem~\ref{thm151210c}\eqref{thm151210c1} since we have
$\supp_R(\Lotimes KY)\subseteq\supp_R(K)=\VE(\fa)$.
The last step is from Fact~\ref{lem150604a2}\eqref{lem150604a2z}.
\end{proof}

\begin{disc}\label{disc151213a}
Note that we do not bother to state the special case $F=\RG aR$ of Theorem~\ref{thm151210d},
in contrast with Corollary~\ref{cor151213b}.
The reason is that Theorem~\ref{thm151210d} is trivial in this case. Indeed, the assumption $\cosupp_R(Y)\subseteq\VE(\fa)$ implies that
$Y\simeq\LL aY\simeq\Rhom{\RG aR}Y$ by Fact~\ref{cor130528a}. Thus, 
for instance Theorem~\ref{thm151210d}\eqref{thm151210d2}
says that for  $*\in\{+,-,\text{b}\}$, one has  $Y\in\catd_*(R)$ if and only if $Y\in\catd_*(R)$.
\end{disc}

\subsection*{Tor Revisited}
Our next result improves upon Theorem~\ref{thm151210b} and contains Theorem~\ref{thm151210ew} from the introduction. 

\begin{thm}\label{thm151210e}
Let $F\in\catdb(R)$ be $\fa$-adically finite such that  $\fd_R(F)<\infty$ and $\supp_R(F)=\VE(\fa)$.
Let $Z\in\catd(R)$ be such that 
$\supp_R(Z)\subseteq\VE(\fa)$.
\begin{enumerate}[\rm(a)]
\item\label{thm151210e1}
There are inequalities
\begin{gather*}
\inf(\Lotimes FZ)\leq\inf(Z)+\sup(F)+2n\\
\inf(F)+\sup(Z)-3n\leq\sup(\Lotimes FZ)\\
\amp(Z)-\amp(F)-5n\leq\amp(\Lotimes FZ).
\end{gather*}
\item\label{thm151210e2}
For  $*\in\{+,-,\text{b}\}$, one has  $Z\in\catd_*(R)$ if and only if $\Lotimes FZ\in\catd_*(R)$.
\end{enumerate}
\end{thm}

\begin{proof}
Recall that~\cite[Theorem~6.1]{sather:afcc} implies $\pd_R(F)<\infty$. 
Let $E$ be a faithfully injective $R$-module, and set $(-)^\vee:=\Rhom -E$.
It follows from Fact~\ref{cor130528a} that $\cosupp_R(Z^\vee)\subseteq\supp_R(Z)\subseteq\VE(\fa)$.
Thus, the third step in the next display is from Theorem~\ref{thm151210d}\eqref{thm151210d1}.
\begin{align*}
\sup(\Lotimes FZ)
&=-\inf((\Lotimes FZ)^\vee)\\
&=-\inf(\Rhom F{Z^\vee}) \\
&\geq \inf(F)-\inf(Z^\vee)-3n\\
&\geq \inf(F)+\sup(Z)-3n
\end{align*}
The first and last steps follow from the fact that $E$ is faithfully injective,
and the second step is from Hom-tensor adjointness.
This explains the second inequality of part~\eqref{thm151210e1}.
The rest of the result follows from Theorem~\ref{thm151210b}.
\end{proof}

To conclude this section, we improve on Theorem~\ref{thm151112b}; the corollary is the special case $F=\RG aR$.

\begin{thm}\label{thm151210f}
Let $F\in\catdb(R)$ be $\fa$-adically finite such that  $\fd_R(F)<\infty$ and $\supp_R(F)=\VE(\fa)$.
Let $Z\in\catd(R)$ be such that 
$\cosupp_R(Z)\subseteq\VE(\fa)$.
\begin{enumerate}[\rm(a)]
\item\label{thm151210f1}
There are inequalities
\begin{gather*}
\inf(\Lotimes FZ)\leq\inf(Z)+\sup(F)+n\\
\sup(Z)+\inf(F)-5n\leq\sup(\Lotimes FZ)\\
\amp(Z)-\amp(F)-6n\leq\amp(\Lotimes FZ).
\end{gather*}
\item\label{thm151210f2}
For  $*\in\{+,-,\text{b}\}$, one has  $Z\in\catd_*(R)$ if and only if $\Lotimes FZ\in\catd_*(R)$.
\end{enumerate}
\end{thm}

\begin{proof}
In light of Theorem~\ref{thm151112b}, we need only bound $\sup(\Lotimes FZ)$.
We compute, starting with an application of Fact~\ref{lem150604a1}\eqref{lem150604a1z}.
\begin{align*}
\sup(\Lotimes FZ)
&=\sup(\Lotimes F{\Lotimes ZK})-n \\
&\geq \inf(F)+\sup(\Lotimes ZK)-4n\\
&\geq \inf(F)+\sup(Z)-5n
\end{align*}
The remaining steps are by Theorem~\ref{thm151210e}\eqref{thm151210e1} and Fact~\ref{lem150604a2}\eqref{lem150604a2z}.
\end{proof}

\begin{cor}\label{cor151213a}
Let $Z\in\catd(R)$ be such that 
$\cosupp_R(Z)\subseteq\VE(\fa)$.
\begin{enumerate}[\rm(a)]
\item\label{cor151213a1}
There are inequalities
\begin{gather*}
\inf(\RG aZ)\leq\inf(Z)-\depth_\fa(R)+n\\
\sup(Z)-\operatorname{cd}_\fa(R)-5n\leq\sup(\RG aZ)\\
\amp(Z)-\operatorname{cd}_\fa(R)+\depth_\fa(R)-6n\leq\amp(\RG aZ).
\end{gather*}
\item\label{cor151213a2}
For  $*\in\{+,-,\text{b}\}$, one has  $Z\in\catd_*(R)$ if and only if $\RG aZ\in\catd_*(R)$.
\item\label{cor151213a3}
One has $Z\simeq 0$ if and only if $\RG aZ\simeq 0$.
\end{enumerate}
\end{cor}

\section{Bounding Homology II: Over a Ring Homomorphism} \label{sec151210a}

We continue with the theme from the previous section.
For instance, the next result, containing Theorem~\ref{thm151128c} from the introduction,
is similar to Theorem~\ref{thm151210e}, 
the differences being fewer restrictions on $F$ and more restrictions on $X$.
For perspective on the assumption $\vf^*(\supp_S(F))\supseteq\VE(\fa)\bigcap\mspec(R)$, 
note that~\cite[Proposition~5.6]{sather:afcc} implies that this condition
is less restrictive than the seemingly more natural condition
$\supp_R(F)\supseteq\VE(\fa)\bigcap\mspec(R)$.
Also, technically, the complex $\Lotimes XF$ in the next result should be written $\Lotimes X{Q(F)}$ where
$Q\colon\catd(S)\to\catd(R)$ is the forgetful functor, and similarly for $\fd_R(F)$.
We choose not to do this, in order to avoid cumbersome notation.

\begin{thm}\label{thm151105b}
Let $\vf\colon R\to S$ be a ring homomorphism such that $\fa S\neq S$, and let $F\in\catdb(S)$ be $\fa S$-adically finite such that  $\fd_R(F)<\infty$  
and $\vf^*(\supp_S(F))\supseteq\VE(\fa)\bigcap\mspec(R)$.
Let $X\in\catd(R)$ be such that  $\supp_R(X)\subseteq\VE(\fa)$.
Assume that
$\Lotimes{K}X\in\catdf(R)$.
\begin{enumerate}[\rm(a)]
\item\label{thm151105b1}
There are inequalities
\begin{gather}
\inf(\Lotimes FX)\leq\inf(X)+\sup(F)+2n \label{thm151105b1a}\\
\sup(X)+\inf(F)-n\leq\sup(\Lotimes FX)\label{thm151105b1b}\\
\amp(X)-\amp(F)-3n\leq\amp(\Lotimes FX).\notag
\end{gather}
\item\label{thm151105b2}
One has $X\simeq 0$ if and only if $\Lotimes FX\simeq 0$.
\item\label{thm151105b3}
For each $*\in\{+,-,\text{b}\}$, one has $X\in\catd_*(R)$ if and only if $\Lotimes FX\in\catd_*(R)$.
\end{enumerate}
\end{thm}

\begin{proof}
\eqref{thm151105b1}
We first treat the case where 
$F\in\catdfb(S)$ and $X\in\catdf(R)$.
Since $F$ is in $\catdfb(S)$, we have $\supp_S(F)=\Supp_S(F)$. 
Also, by~\cite[Proposition~3.15(a)]{sather:scc} we know that $\Supp_R(X)\subseteq\VE(\fa)$.
Thus, from the proof of~\cite[Theorem~4.2]{frankild:rrhffd}, we have
\begin{gather*}
\inf(\Lotimes FX)\leq\inf(X)+\sup(F)\\
\sup(X)+\inf(F)\leq\sup(\Lotimes FX)\\
\amp(X)-\amp(F)\leq\amp(\Lotimes FX).
\end{gather*}

For the general case, set $K^S:=\Lotimes S{K}$. 
We apply the previous case to the complexes $\Lotimes[S]{K^S}F$ and $\Lotimes{K}X$.
For this, we need to check the support conditions for these complexes.
For the first one, we have
$$\supp_S(\Lotimes[S]{K^S}{F})=\VE(\fa S)\bigcap\supp_S(F)$$
by Fact~\ref{cor130528a},
so the condition $\vf^*(\supp_S(F))\supseteq\VE(\fa)\bigcap\mspec(R)$
implies that  
$$\vf^*(\supp_S(\Lotimes[S]{K^S}{F}))\supseteq\VE(\fa)\bigcap\mspec(R).$$
The condition $\supp_R(\Lotimes{K}X)\subseteq\VE(\fa)$ is even easier.

The first inequality in the next sequence is from Fact~\ref{disc151112a}\eqref{disc151112a5}.
\begin{align*}
\inf(\Lotimes FX)
&\leq\inf(\Lotimes[S]{K^S}{\Lotimes{(\Lotimes FX)}{K}})\\
&=\inf(\Lotimes{(\Lotimes[S]{K^S}F)}{(\Lotimes{K}X)})\\
&\leq\sup(\Lotimes[S]{K^S}F)+\inf(\Lotimes{K}X)\\
&\leq 2n+\sup(F)+\inf(X)
\end{align*}
The second step is from the isomorphism
$$\Lotimes[S]{K^S}{\Lotimes{(\Lotimes FX)}{K}}\simeq\Lotimes{(\Lotimes[S]{K^S}F)}{(\Lotimes{K}X)}$$
and the third step is from the first paragraph of this proof.
The fourth step is a combination of Facts~\ref{disc151112a}\eqref{disc151112a5} and~\ref{lem150604a1}\eqref{lem150604a1z}.
This explains the inequality~\eqref{thm151105b1a}, while~\eqref{thm151105b1b} follows similarly:
\begin{align*}
\inf(F)+n+\sup(X)
&\leq\inf(\Lotimes[S]{K^S}F)+\sup(\Lotimes{K}X)\\
&\leq\sup(\Lotimes{(\Lotimes[S]{K^S}F)}{(\Lotimes{K}X)})\\
&=\sup(\Lotimes[S]{K^S}{\Lotimes{(\Lotimes FX)}{K}})\\
&\leq 2n+\sup(\Lotimes FX).
\end{align*}
The remaining conclusions from the statement of the theorem follow readily.
\end{proof}

Next, we have a version of Theorem~\ref{thm151105b} for co-support.
It is proved like Theorem~\ref{thm151105b}, using 
Fact~\ref{lem150604a2}\eqref{lem150604a2z} instead of~\ref{lem150604a1}\eqref{lem150604a1z}.

\begin{thm}\label{thm151211a}
Let $\vf\colon R\to S$ be a ring homomorphism such that $\fa S\neq S$, and let $F\in\catdb(S)$ be $\fa S$-adically finite such that  $\fd_R(F)<\infty$  
and $\vf^*(\supp_S(F))\supseteq\VE(\fa)\bigcap\mspec(R)$.
Let $X\in\catd(R)$ be such that  $\cosupp_R(X)\subseteq\VE(\fa)$.
Assume that
$\Lotimes{K}X\in\catdf(R)$.
\begin{enumerate}[\rm(a)]
\item\label{thm151211a1}
There are inequalities
\begin{gather*}
\inf(\Lotimes FX)\leq\inf(X)+\sup(F)+n\\
\sup(X)+\inf(F)-3n\leq\sup(\Lotimes FX)\\
\amp(X)-\amp(F)-4n\leq\amp(\Lotimes FX)\leq\amp(X)+\fd_R(F)-\inf(F).
\end{gather*}
\item\label{thm151211a2}
One has $X\simeq 0$ if and only if $\Lotimes FX\simeq 0$.
\item\label{thm151211a3}
For each $*\in\{+,-,\text{b}\}$, one has $X\in\catd_*(R)$ if and only if $\Lotimes FX\in\catd_*(R)$.
\end{enumerate}
\end{thm}

\subsection*{Ext Re-revisited}
The previous two theorems yield more results about $\mathbf{R}\!\operatorname{Hom}$.
We begin with two results for $P\in\catdfb(R)$. 
These are  extended to the adically finite setting in Theorems~\ref{cor151114a} and~\ref{thm151211c} below.

\begin{cor}\label{cor151114b}
Let $P\in\catdfb(R)$ be such that $\supp_R(P)\supseteq\VE(\fa)\bigcap\mspec(R)$  
and  $\pd_R(P)<\infty$.
Let $X\in\catd(R)$ be such that  $\supp_R(X)\subseteq\VE(\fa)$ and $\Lotimes{K}X\in\catdf(R)$.
\begin{enumerate}[\rm(a)]
\item\label{cor151114b1}
There are inequalities
\begin{gather*}
\inf(\Rhom PX)\leq\inf(X)-\inf(P)+2n\\
\sup(X)-\pd_R(P)-n\leq\sup(\Rhom PX)\\
\amp(X)-\pd_R(P)+\inf(P)-3n\leq\amp(\Rhom PX).
\end{gather*}
\item\label{cor151114b2}
One has $X\simeq 0$ if and only if $\Rhom PX\simeq 0$.
\item\label{cor151114b3}
For  $*\in\{+,-,\text{b}\}$, one has $X\in\catd_*(R)$ if and only if $\Rhom PX\in\catd_*(R)$.
\end{enumerate}
\end{cor}

\begin{proof}
Set $P^*:=\Rhom PR$.
The  assumptions $P\in\catdfb(R)$ and $\pd_R(P)<\infty$ imply that $P\simeq\Rhom{P^*}R$.
In particular, we have $P\simeq 0$ if and only if $P^*\simeq 0$. 
For any prime ideal $\p\in\spec(R)$, it follows that $P_{\p}\simeq 0$ if and only if $(P^*)_{\p}\simeq 0$,
so we have $\Supp_R(P)=\Supp_R(P^*)$, that is, $\supp_R(P)=\supp_R(P^*)$ since $P,P^*\in\catdfb(R)$.

In the next sequence of isomorphisms, the second step is Hom-evaluation~\cite[Lemma~4.4(I)]{avramov:hdouc}
$$\Rhom{P}{X}\simeq\Rhom{\Rhom{P^*}R}{X}\simeq\Lotimes{P^*}{\Rhom RX}\simeq\Lotimes{P^*}X$$
and the other steps are routine.
In light of the next (in)equalities
\begin{gather*}
\inf(P^*)=-\pd_R(P)\\
\sup(P^*)\leq\fd_R(P^*)=\pd_R(P^*)=-\inf(P)
\end{gather*}
the desired conclusions follow from Theorem~\ref{thm151105b}, with $\vf=\id_R\colon R\to R$.
\end{proof}

\begin{cor}\label{cor151211a}
Let $P\in\catdfb(R)$ be such that $\supp_R(P)\supseteq\VE(\fa)\bigcap\mspec(R)$  
and  $\pd_R(P)<\infty$.
Let $X\in\catd(R)$ be such that  $\cosupp_R(X)\subseteq\VE(\fa)$ and $\Lotimes{K}X\in\catdf(R)$.
\begin{enumerate}[\rm(a)]
\item\label{cor151211a1}
There are inequalities
\begin{gather*}
\inf(\Rhom PX)\leq\inf(X)-\inf(P)+n\\
\sup(X)-\pd_R(P)-3n\leq\sup(\Rhom PX)\\
\amp(X)-\pd_R(P)+\inf(P)-4n\leq\amp(\Rhom PX).
\end{gather*}
\item\label{cor151211a2}
One has $X\simeq 0$ if and only if $\Rhom PX\simeq 0$.
\item\label{cor151211a3}
For  $*\in\{+,-,\text{b}\}$, one has $X\in\catd_*(R)$ if and only if $\Rhom PX\in\catd_*(R)$.
\end{enumerate}
\end{cor}

\begin{proof}
Argue as for Corollary~\ref{cor151114b}, using Theorem~\ref{thm151211a}.
\end{proof}

Again, in contrast with the results of Section~\ref{sec151104a}, the point of the next results is that we 
assume less for $P$ but more for $X$.

\begin{thm}\label{cor151114a}
Let $P\in\catdb(R)$ be $\fa$-adically finite such that  $\pd_R(P)<\infty$  
and $\supp_R(P)\supseteq\VE(\fa)\bigcap\mspec(R)$.
Let $X\in\catd(R)$ be such that  $\supp_R(X)\subseteq\VE(\fa)$.
Assume that
$\Lotimes{K}X\in\catdf(R)$.
\begin{enumerate}[\rm(a)]
\item\label{cor151114a1}
There are inequalities
\begin{gather}
\inf(\Rhom PX)\leq\inf(X)-\inf(P)+3n \label{cor151114a1a} \\
\sup(X)-\pd_R(P)-2n\leq\sup(\Rhom PX) \label{cor151114a1b}\\
\amp(X)+\inf(P)-\pd_R(P)-5n\leq\amp(\Rhom PX).\notag
\end{gather}
\item\label{cor151114a2}
One has $X\simeq 0$ if and only if $\Rhom PX\simeq 0$.
\item\label{cor151114a3}
For  $*\in\{+,-,\text{b}\}$, one has $X\in\catd_*(R)$ if and only if $\Rhom PX\in\catd_*(R)$.
\end{enumerate}
\end{thm}

\begin{proof}
We verify~\eqref{cor151114a1a} and~\eqref{cor151114a1b},
using the  Hom-evaluation isomorphism
\begin{align}
\Lotimes K{\Rhom PX}
&\simeq\Rhom {\Rhom KP}X \label{eq151129a}
\end{align} 
in $\catd(R)$; see Fact~\ref{disc151211a}. This explains the second step in the next sequence.
\begin{align*}
\inf(\Rhom PX)
&\leq\inf(\Lotimes K{\Rhom PX}) \\
&=\inf(\Rhom{\Rhom KP}X)
\end{align*} 
The first step is by Fact~\ref{disc151112a}\eqref{disc151112a5}.

Note that we have $\Rhom KP\in\catdfb(R)$, by assumption.
Also, we have the self-dual isomorphism
\begin{equation}\label{eq151115a}
\Rhom KP\simeq\Lotimes{\shift^{-n}K}P
\end{equation}
from Fact~\ref{disc151211a}.
Thus, in the next sequence, the first  step is from Fact~\ref{cor130528a}.
\begin{align*}
\supp_R(\Rhom KP) 
&=\supp_R(K)\bigcap\supp_R(P) 
\supseteq\VE(\fa)\bigcap\mspec(R)
\end{align*} 
The last step is by assumption, since $\supp_R(K)=\VE(\fa)$.

Now we compute. 
The first step in the next sequence is from~\eqref{eq151115a}
\begin{align*}
\inf(\Rhom KP)
&=\inf(\Lotimes{\shift^{-n}K}P) 
\geq\inf(\shift^{-n}K)+\inf(P) 
=-n+\inf(P).
\end{align*}
The second step is from Fact~\ref{disc151112a}\eqref{disc151112a5}, and the third follows from the equality $\inf(K)=0$.
This explains the third inequality in the next sequence
\begin{align*}
\inf(\Rhom PX)
&\leq\inf(\Rhom{\Rhom KP}X)\\
&\leq\inf(X)-\inf(\Rhom KP)+2n \\
&\leq\inf(X)-\inf(P)+3n.
\end{align*}
The second inequality here is from Corollary~\ref{cor151114b}, and the first one is from the first paragraph of this proof.
This establishes~\eqref{cor151114a1a}.

For~\eqref{cor151114a1b}, the first step in the next sequence is from~\eqref{eq151115a}.
\begin{align*}
\pd_R(\Rhom KP)
&=\pd_R(\Lotimes{\shift^{-n} K}P) 
\leq\pd(\shift^{-n} K)+\pd_R(P) 
=\pd_R(P)
\end{align*}
The second step is from~\cite[Theorem 4.1(P)]{avramov:hdouc},
and the third step is from the equality $\pd_R(\shift^{-n} K)=0$.
This explains the last step in the next sequence.
\begin{align*}
\sup(\Rhom PX)
&\geq\sup(\Lotimes K{\Rhom PX})-n\\
&=\sup(\Rhom{\Rhom KP}X)-n\\
&\geq\sup(X)-\pd_R(\Rhom KP)-2n \\
&\geq\sup(X)-\pd_R(P)-2n
\end{align*}
The first step is from Fact~\ref{disc151112a}\eqref{disc151112a5}, the second one is is from~\eqref{eq151129a}, 
and the third one is from Corollary~\ref{cor151114b}.
\end{proof}

\begin{thm}\label{thm151211c}
Let $P\in\catdb(R)$ be $\fa$-adically finite such that  $\pd_R(P)<\infty$  
and $\supp_R(P)\supseteq\VE(\fa)\bigcap\mspec(R)$.
Let $X\in\catd(R)$ be such that  $\cosupp_R(X)\subseteq\VE(\fa)$.
Assume that
$\Lotimes{K}X\in\catdf(R)$.
\begin{enumerate}[\rm(a)]
\item\label{thm151211c1}
There are inequalities
\begin{gather*}
\inf(\Rhom PX)\leq\inf(X)-\inf(P)+2n \\
\sup(X)-\pd_R(P)-4n\leq\sup(\Rhom PX) \\
\amp(X)+\inf(P)-\pd_R(P)-6n\leq\amp(\Rhom PX).
\end{gather*}
\item\label{thm151211c2}
One has $X\simeq 0$ if and only if $\Rhom PX\simeq 0$.
\item\label{thm151211c3}
For  $*\in\{+,-,\text{b}\}$, one has $X\in\catd_*(R)$ if and only if $\Rhom PX\in\catd_*(R)$.
\end{enumerate}
\end{thm}

\begin{proof}
Argue as for Theorem~\ref{cor151114a}, using Corollary~\ref{cor151211a} in place of~\ref{cor151114b}.
\end{proof}

\section{Bounding Homology III: Modules}\label{sec151211b}

The next  results show how one can replace a faithfulness hypothesis with appropriate support assumptions,
beginning with the projective situation.

\begin{prop}
\label{prop151108a}
Let $P$ be a projective $R$-module such that $\VE(\fa)\bigcap\mspec(R)\subseteq\supp_R(P)$.
Let $Y\in\catd(R)$ be such that  $\supp_R(Y)\subseteq\VE(\fa)$ or $\cosupp_R(Y)\subseteq\VE(\fa)$.
\begin{enumerate}[\rm(a)]
\item\label{prop151108a1}
Then we have $\Supp_R(P)=\supp_R(P)\supseteq\VE(\fa)$.
\item\label{prop151108a2}
There is a projective $R$-module $Q$ such that $P\oplus Q$ is faithfully projective
and $\Rhom{P\oplus Q}Y\simeq\Rhom PY$.
\item\label{prop151108a5}
One has $\HH_i(\Rhom PY)\neq 0$ if and only if $\HH_i(Y)\neq 0$, so
there are equalities
\begin{gather*}
\inf(\Rhom PY)=\inf (Y) \\
\sup(\Rhom PY)=\sup (Y) \\
\amp(\Rhom PY)=\amp (Y). 
\end{gather*}
\item\label{prop151108a3}
One has $Y\simeq 0$ if and only if $\Rhom PY\simeq 0$.
\item\label{prop151108a4}
For  $*\in\{+,-,\text{b}\}$, one has $Y\in\catd_*(R)$ if and only if $\Rhom PY\in\catd_*(R)$.
\end{enumerate}
\end{prop}

\begin{proof}
Being projective, the module $P$ is locally free.
It follows readily that we have $\Supp_R(P)=\supp_R(P)$.
This also implies that the characteristic function $f\colon\spec(R)\to\{0,1\}$ for $\Supp_R(P)$, given by the formula
$$
f(\p):=\begin{cases}
1&\text{if $P_{\p}\neq 0$} \\
0&\text{if $P_{\p}= 0$} \end{cases}$$
is locally constant. 
Thus, there is a decomposition $R\cong R'\times R''$ such that, under the canonical identification
$\spec(R)=\spec(R')\bigsqcup\spec(R'')$, we have $\spec(R')=\Supp_R(P)$.
From this, it follows that we have $P\cong P'\times 0$ for some projective $R'$-module $P'$ with
$$\Supp_{R'}(P')=\Supp_R(P)=\spec(R').$$
(We obtain $P'$ as the localization of $P$ at the idempotent $e_2=(0,1)$.)
In particular, $P'$ is faithfully projective over $R'$.

\eqref{prop151108a1}
It remains to show that $\Supp_R(P)\supseteq\VE(\fa)$, so let $\p\in\VE(\fa)$.
Fix a maximal ideal $\m\supseteq\p\supseteq\fa$, so we have 
$$\m\in\VE(\fa)\bigcap\mspec(R)\subseteq\supp_R(P)=\Supp_R(P)=\spec(R').$$
The decomposition  $\spec(R)=\spec(R')\bigsqcup\spec(R'')$ says that we have $\m=\m'\times R''$ for some
maximal ideal $\m'\in\mspec(R')$.
Thus, the containment $\p\subseteq\m$ implies 
that we have $\p=\p'\times R''$ for some
prime ideal $\p'\in\spec(R')$; in other words, we have $\p\in\spec(R')=\Supp_R(P)$, as desired.

\eqref{prop151108a2}
We have already seen that $P\cong P'\times 0$ for some faithfully projective $R'$-module $P'$.
Set $Q:=0\times R''$. It is straightforward to show that $P\oplus Q\cong P'\times R''$ is faithfully projective.

Assume for this paragraph that we have $\cosupp_R(Y)\subseteq\VE(\fa)$.
We prove that this implies that $\supp_R(Y)\subseteq\supp_R(P)$.
The key point here is from~\cite[Proposition~4.7(b)]{sather:scc} which says that
$\supp_R(Y)$ and $\cosupp_R(Y)$ have the same maximal elements with respect to containment.
Let $\p\in\supp_R(Y)$, and let $\q\supseteq\p$ be maximal in $\supp_R(Y)$ with respect to containment.
It follows that 
$$\q\in\cosupp_R(Y)\subseteq\VE(\fa)\subseteq\supp_R(P)=\spec(R').$$
As in the proof of part~\eqref{prop151108a1}, the condition $\p\subseteq\q$ implies  that $\p\in\supp_R(P)$.

In the alternate case $\supp_R(Y)\subseteq\VE(\fa)$, part~\eqref{prop151108a1} implies that 
$\supp_R(Y)\subseteq\supp_R(P)$.
Thus, we assume for the rest of the proof that  $\supp_R(Y)\subseteq\supp_R(P)=\spec(R')=\VE(e_2)$. It follows that 
$\Supp_R(Y)\subseteq\VE(e_2)
=\spec(R')$
by Fact~\ref{fact151129a}.
Again localizing at $e_2$, we conclude that $Y\simeq Y'\times 0$ for some $Y'\in\catd(R')$.
From this we have the following sequence which gives the desired conclusion:
\begin{align*}
\Rhom{P\oplus Q}{Y}
&\simeq\Rhom[R'\times R'']{P'\times R''}{Y'\times 0}\\
&\simeq\Rhom[R']{P'}{Y'}\times\Rhom[R'']{R''}{0}\\
&\simeq\Rhom[R']{P'}{Y'}\times 0\\
&\simeq\Rhom[R']{P'}{Y'}\times\Rhom[R'']{0}{0}\\
&\simeq\Rhom[R'\times R'']{P'\times 0}{Y'\times 0}\\
&\simeq\Rhom{P}{Y}.
\end{align*}

\eqref{prop151108a5}--\eqref{prop151108a4}
The fact that $P\oplus Q$ is faithfully projective over $R$ implies that
$\HH_i(Y)\neq 0$ if and only if $\HH_i(\Rhom{P\oplus Q}Y)\neq 0$,
that is, if and only if $\HH_i(\Rhom{P}Y)\neq 0$.
The desired conclusions now follow directly.
\end{proof}

The next two results show how faithful flatness can be relaxed.

\begin{prop}
\label{prop151114a}
Let $F$ be a flat $R$-module, and 
let $Y\in\catd(R)$ be such that we have $\Supp_R(Y)\bigcap\mspec(R)\subseteq\supp_R(F)$,
e.g., such that $\supp_R(Y)\subseteq\VE(\fa)$ and $\VE(\fa)\bigcap\mspec(R)\subseteq\supp_R(F)$.
\begin{enumerate}[\rm(a)]
\item\label{prop151114a1}
One has $\HH_i(\Lotimes FY)\neq 0$ if and only if $\HH_i(Y)\neq 0$, so
there are equalities
\begin{gather*}
\inf(\Lotimes FY)=\inf (Y) \\
\sup(\Lotimes FY)=\sup (Y) \\
\amp(\Lotimes FY)=\amp (Y). 
\end{gather*}
\item\label{prop151114a2}
One has $Y\simeq 0$ if and only if $\Lotimes FY\simeq 0$.
\item\label{prop151114a3}
For each $*\in\{+,-,\text{b}\}$, one has $Y\in\catd_*(R)$ if and only if $\Lotimes FY\in\catd_*(R)$.
\end{enumerate}
\end{prop}

\begin{proof}
Note that if $\supp_R(Y)\subseteq\VE(\fa)$, then $\Supp_R(Y)\subseteq\VE(\fa)$ by~\cite[Proposition~3.15(a)]{sather:scc};
if we also have $\VE(\fa)\bigcap\mspec(R)\subseteq\supp_R(F)$, then it follows that
$\Supp_R(Y)\bigcap\mspec(R)\subseteq\VE(\fa)\bigcap\mspec(R)\subseteq\supp_R(F)$. 

As in the proof of Proposition~\ref{prop151108a}, it suffices to show that for each $i\in\bbz$ one has
$\HH_i(\Lotimes FY)\neq 0$ if and only if $\HH_i(Y)\neq 0$.
Since $F$ is a flat $R$-module, we have $\HH_i(\Lotimes FY)\cong\Otimes F{\HH_i(Y)}$,
so the forward implication in the previous sentence is routine. For the converse, assume that $\HH_i(Y)\neq 0$,
and let 
$$\m\in\Supp_R(\HH_i(Y))\bigcap\mspec(R)\subseteq\Supp_R(Y)\bigcap\mspec(R)\subseteq\supp_R(F).$$
It follows that the flat $R_\m$-module $F_\m$ has $\m\in\supp_{R_\m}(F_\m)$;
in other words, $F_\m$ is faithfully flat over $R_\m$. 
By assumption, we have $\HH_i(Y)_\m\neq 0$, so
$$
0\neq
\Otimes[R_\m]{F_\m}{\HH_i(Y)_\m}
\cong\HH_i(\Otimes FY)_\m.$$
We conclude that $\HH_i(\Otimes FY)\neq 0$, as desired.
\end{proof}

\begin{prop}
\label{prop151211a}
Let $F$ be a flat $R$-module with $\VE(\fa)\bigcap\mspec(R)\subseteq\supp_R(F)$.
Let $Y\in\catd(R)$ be 
such that  $\cosupp_R(Y)\subseteq\VE(\fa)$.
\begin{enumerate}[\rm(a)]
\item\label{prop151211a1}
There are (in)equalities
\begin{gather*}
\inf(\Lotimes FY)=\inf (Y) \\
\sup(Y)-2n\leq\sup(\Lotimes FY) \\
\amp (Y)-2n\leq\amp(\Lotimes FY). 
\end{gather*}
\item\label{prop151211a2}
One has $Y\simeq 0$ if and only if $\Lotimes FY\simeq 0$.
\item\label{prop151211a3}
For each $*\in\{+,-,\text{b}\}$, one has $Y\in\catd_*(R)$ if and only if $\Lotimes FY\in\catd_*(R)$.
\end{enumerate}
\end{prop}

\begin{proof}
For the infimum computation, we begin with two applications of Fact~\ref{disc151112a}\eqref{disc151112a5}.
\begin{align*}
\inf(Y)
&\leq\inf(\Lotimes FY) 
\leq\inf(\Lotimes F{\Lotimes YK}) 
=\inf(\Lotimes YK) 
=\inf(Y)
\end{align*}
The third step here is from Proposition~\ref{prop151114a}\eqref{prop151114a1}, using the condition $\supp_R(\Lotimes KY)\subseteq\VE(\fa)$.
The fourth step is from Fact~\ref{lem150604a2}\eqref{lem150604a2z}.

For the supremum bound, we argue similarly
\begin{align*}
\sup(\Lotimes FY) 
&\geq\sup(\Lotimes F{\Lotimes YK})-n \\
&=\sup(\Lotimes YK)-n \\
&=\sup(Y)-2n
\end{align*}
using 
Fact~\ref{disc151112a}\eqref{disc151112a5}, Proposition~\ref{prop151114a}\eqref{prop151114a1}, and Fact~\ref{lem150604a2}\eqref{lem150604a2z}.
\end{proof}

For perspective in the injective versions of this section, we recall the following.
Given an injective $R$-module $I\cong\bigoplus_{\p\in\spec(R)}E_R(R/\p)^{(\mu_\p)}$,
one has
\begin{align*}
\supp_R(I)&=\{\p\in\spec(R)\mid\mu_\p\neq\emptyset\} \\
\cosupp_R(I)&=\{\q\in\spec(R)\mid\text{there is a $\p\in\spec(R)$ such that $\q\subseteq\p$ and $\mu_\p\neq\emptyset$}\}
\end{align*}
by~\cite[Propositions~3.8 and~6.3]{sather:scc}. 

\begin{prop}
\label{prop151114b}
Let $I$ be an injective $R$-module
with $\VE(\fa)\subseteq\cosupp_R(I)$, e.g., with $\VE(\fa)\bigcap\mspec(R)\subseteq\supp_R(I)$.
Let $Y\in\catd(R)$ be such that $\supp_R(Y)\subseteq\VE(\fa)$.
\begin{enumerate}[\rm(a)]
\item\label{prop151114b1}
One has $\HH_i(\Rhom YI)\neq 0$ if and only if $\HH_{-i}(Y)\neq 0$, so
there are equalities
\begin{gather*}
\inf(\Rhom YI)=-\sup (Y) \\
\sup(\Rhom YI)=-\inf (Y) \\
\amp(\Rhom YI)=\amp (Y). 
\end{gather*}
\item\label{prop151114b2}
One has $Y\simeq 0$ if and only if $\Rhom YI\simeq 0$.
\item\label{prop151114b3}
For  $*\in\{+,-,\text{b}\}$, one has $Y\in\catd_*(R)$ if and only if $\Rhom YI\in\catd_*(R)$.
\end{enumerate}
\end{prop}

\begin{proof}
Assume first that $\VE(\fa)\subseteq\cosupp_R(I)$.
Again, it suffices to show that for each $i\in\bbz$ one has
$\HH_i(\Rhom YI)\neq 0$ if and only if $\HH_{-i}(Y)\neq 0$.
Since $I$ is an injective $R$-module, we have 
$$\HH_i(\Rhom YI)\cong\Hom{\HH_{-i}(Y)}I\simeq\Rhom{\HH_{-i}(Y)}I$$
so the forward implication in the previous sentence is routine. For the converse, assume that $\HH_{-i}(Y)\neq 0$,
and let $\p\in\supp_R(\HH_{-i}(Y))$.
Note that the condition $\supp_R(Y)\subseteq\VE(\fa)$ implies that we have
$$\supp_R(\HH_{-i}(Y))\subseteq\Supp_R(\HH_{-i}(Y))\subseteq\Supp_R(Y)\subseteq\VE(\fa)$$
by Fact~\ref{fact151129a}.
It follows that
$$\p\in\supp_R(\HH_{-i}(Y))\subseteq\VE(\fa)\subseteq\cosupp_R(I).$$
We conclude  that $\p\in\supp_R(\HH_{-i}(Y))\bigcap\cosupp_R(I)$,
so  $\Rhom{\HH_{-i}(Y)}I\not\simeq 0$ by Fact~\ref{cor130528a}, as desired.

Assume next that $\VE(\fa)\bigcap\mspec(R)\subseteq\supp_R(I)$.
We need to show that $\VE(\fa)\subseteq\cosupp_R(I)$, so let $\p\in\VE(\fa)$.
Fix a maximal ideal $\m$ of $R$ such that $\m\supseteq\p\supseteq\fa$,
so we have $\m\in\VE(\fa)\bigcap\mspec(R)\subseteq\supp_R(I)\subseteq\cosupp_R(I)$
by the notes preceding the statement of this result.
Since we have $\p\subseteq\m\in\cosupp_R(I)$, these notes also imply that $\p\in\cosupp_R(I)$,  as desired.
\end{proof}

\begin{prop}\label{prop151211b}
Let $I$ be an injective $R$-module
with $\VE(\fa)\subseteq\cosupp_R(I)$, e.g., with $\VE(\fa)\bigcap\mspec(R)\subseteq\supp_R(I)$.
Let $Y\in\catd(R)$ be with $\cosupp_R(Y)\subseteq\VE(\fa)$.
\begin{enumerate}[\rm(a)]
\item\label{prop151211b1}
There are (in)equalities
\begin{gather}
\inf(\Rhom YI)\leq 2n-\sup (Y) \label{prop151211b1a}\\
\sup(\Rhom YI)=-\inf (Y) \label{prop151211b1b}\\
\amp(Y)-2n\leq\amp(\Rhom YI). \notag 
\end{gather}
\item\label{prop151211b2}
One has $Y\simeq 0$ if and only if $\Rhom YI\simeq 0$.
\item\label{prop151211b3}
For  $*\in\{+,-,\text{b}\}$, one has $Y\in\catd_*(R)$ if and only if $\Rhom YI\in\catd_*(R)$.
\end{enumerate}
\end{prop}

\begin{proof}
We argue as for Proposition~\ref{prop151211a}, using Proposition~\ref{prop151114b}.
In the next sequence, the first and last steps are from 
Fact~\ref{disc151112a}\eqref{disc151112a5}.
\begin{align*}
\inf(\Rhom YI) 
&\leq\inf(\Lotimes K{\Rhom YI}) \\
&=\inf(\Rhom {\Rhom KY}I) \\
&=-\sup(\Rhom KY) \\
&=-\sup(\shift^{-n}\Lotimes KY) \\
&=n-\sup(\Lotimes KY) \\
&\leq2n-\sup(Y)
\end{align*}
The second step here is by Hom-evaluation~\ref{disc151211a}.
The third step is by Proposition~\ref{prop151114b}, since the fact that $\fa$ annihilates $\HH(\Rhom KY)$ implies that we have $\supp_R(\Rhom KY)\subseteq\VE(\fa)$.
The fourth step  is by the self-dual nature~\ref{disc151211a} of $K$, and the fifth step is routine.
This establishes~\eqref{prop151211b1a}.

For~\eqref{prop151211b1b}, we being the next sequence 
with two applications of Fact~\ref{disc151112a}\eqref{disc151112a5}.
\begin{align*}
-\inf(Y)
&\geq\sup(\Rhom YI) \\
&\geq\sup(\Lotimes K{\Rhom YI})-n \\
&=\sup(\Rhom{\Rhom KY}I)-n \\
&=-\inf(\Rhom KY)-n \\
&=-\inf(\shift^{-n}\Lotimes KY)-n\\
&=-\inf(\Lotimes KY)\\
&=-\inf(Y)
\end{align*}
The remaining steps are by  Hom-evaluation~\ref{disc151211a}, Proposition~\ref{prop151114b}, the self-dual nature~\ref{disc151211a} of $K$, 
a routine computation, and Fact~\ref{lem150604a2}\eqref{lem150604a2z}.
\end{proof}

\section{Applications}\label{sec151126a}

We end with an indication of some of the applications of our boundedness results,
following Foxby and Iyengar~\cite{foxby:daafuc}.
See also~\cite{sather:afc} for other applications.

\begin{thm}\label{thm151126a}
Let $(Q,\m_Q,k)\to(R,\m_R)\xra\vf (S,\m_S)$ be  local ring homomorphisms such that $\fa\subseteq\rad{\m_Q R}$.
Assume that $F\in\catdb(S)$ is $\fa S$-adically finite over $S$
with $F\not\simeq 0$ and $\fd_R(F)<\infty$. 
Then there are inequalities
$$\fd_Q(R)+\inf(F)-n\leq\fd_Q(F)\leq\fd_Q(R)+\fd_R(F).$$
In particular, the quantities $\fd_Q(R)$ and $\fd_Q(F)$ are simultaneously finite.
\end{thm}

\begin{proof}
Note that the fact that the map $R\to S$ is local
implies that $\fa S\neq S$. 

The second of our desired inequalities is from~\cite[Corollary~4.2(bF)]{avramov:hdouc}.
To verify the first of our desired inequalities, we argue as in~\cite[Theorem~3.2]{foxby:daafuc},
using Theorem~\ref{thm151105b} in place of~\cite[Theorem~3.1]{foxby:daafuc}.
The first step in the next display is from Fact~\ref{disc151112a}\eqref{disc151112a5}.
\begin{align*}
\fd_Q(F)
&\geq\sup(\Lotimes[Q]kF)\\
&=\sup(\Lotimes{(\Lotimes[Q]kR)}F) \\
&\geq\sup(\Lotimes[Q]kR)+\inf(F)-n \\
&=\fd_Q(R)+\inf(F)-n
\end{align*}
The second step is from tensor cancellation $\Lotimes[Q]kF\simeq\Lotimes{(\Lotimes[Q]kR)}F$,
and the fourth step is~\cite[Proposition~5.5(F)]{avramov:hdouc}.
The third step is from Theorem~\ref{thm151105b}\eqref{thm151105b1}
with $X:=\Lotimes[Q]kR\in\catdf_+(R)$; we need to check that the hypotheses of this result are satisfied.
The condition $F\not\simeq 0$ implies that $\m_S\in\Supp_S(F)=\supp_S(F)$, by Fact~\ref{fact151129a}.
In particular, we have $\m_R=\vf^{-1}(\m_S)$, and so $\VE(\fa)\bigcap\mspec(R)=\{\m_R\}\subseteq\vf^*(\supp_S(F))$.
The homology of $\Lotimes [Q]kR$ is annihilated by $\m_QR$, hence the first containment in the next sequence.
\begin{align*}
\supp_R(\Lotimes[Q]kR)
&\subseteq\VE(\m_QR)
\subseteq\VE(\fa)
\end{align*}
The second containment is from the assumption $\fa\subseteq\rad{\m_Q R}$.
Thus, Theorem~\ref{thm151105b}\eqref{thm151105b1} applies, as desired.
\end{proof}

We conclude with Theorem~\ref{thm151126bz} from the introduction.
Our argument is based in spirit on that of~\cite[Theorem~3.3]{foxby:daafuc}.

\begin{thm}\label{thm151126b}
Let $(R,\m)$ be a local ring of prime characteristic, and let $\vf\colon R\to R$ be the Frobenius endomorphism. 
Then the following conditions are equivalent.
\begin{enumerate}[\rm(i)]
\item\label{thm151126b1} $R$ is regular.
\item\label{thm151126b2} $\vf^t$ is flat for each (equivalently, some) integer $t\geq 1$.
\item\label{thm151126b3} One has $\fd_R({}^{\vf^t}\RG bR)<\infty$ for some integer $t\geq 1$ and some ideal $\fb\subseteq\m$.
\item\label{thm151126b4} One has $\fd_R({}^{\vf^t}\!F)<\infty$ for some integer $t\geq 1$ for some $\fb$-adically finite $R$-complex 
$F\not\simeq 0$ with $\fd_R(F)<\infty$ for some ideal $\fb\subseteq\m$.
\end{enumerate}
\end{thm}

\begin{proof}
The equivalence of conditions~\eqref{thm151126b1} and~\eqref{thm151126b2} are from~\cite[(2.1)]{kunz:corlrocp}.
The implication~\eqref{thm151126b2}$\implies$\eqref{thm151126b3} is from~\cite[Corollary~4.2(bF)]{avramov:hdouc},
since $\RG bR$ has finite flat dimension
(via the \v Cech complex).
The implication~\eqref{thm151126b3}$\implies$\eqref{thm151126b4} is from the fact that $\RG bR$ has finite flat dimension
and is $\fb$-adically finite by~\cite[Theorem~7.10]{sather:scc}.

\eqref{thm151126b4}$\implies$\eqref{thm151126b2}
Assume that $\fd_R({}^{\vf^t}\!F)<\infty$ for some integer $t\geq 1$ for some $\fb$-adically finite $R$-complex 
$F\not\simeq 0$ with $\fd_R(F)<\infty$ for some ideal $\fb\subseteq\m$.
To prove that $\vf^t$ is flat, it suffices by
the proof of~\cite[Theorem~3.3]{foxby:daafuc} to show that $\fd_R({}^{\vf^t} \!R)<\infty$.

Since $\vf$ is the Frobenius endomorphism, we have $\rad{\vf^{t}(\m)R}=\m\supseteq\fb$.
Thus, the hypotheses of Theorem~\ref{thm151126a} are satisfied with the homomorphisms $R\xra{\vf^{t}} R\xra{\id} R$
and $\fa=\fb$.
We conclude that 
$\fd_R({}^{\vf^t} \!R)\leq\fd_R({}^{\vf^t}\!F)-\inf(F)+n<\infty$
as desired.
\end{proof}

\section*{Acknowledgments}
We are grateful to Srikanth Iyengar, 
Liran Shaul,
and Amnon Yekutieli
for helpful comments about this work.

\providecommand{\bysame}{\leavevmode\hbox to3em{\hrulefill}\thinspace}
\providecommand{\MR}{\relax\ifhmode\unskip\space\fi MR }
\providecommand{\MRhref}[2]{%
  \href{http://www.ams.org/mathscinet-getitem?mr=#1}{#2}
}
\providecommand{\href}[2]{#2}

\end{document}